\documentclass{amsart}

\usepackage{amssymb} 
\usepackage{latexsym}
\usepackage{comment}
\usepackage{url}
\usepackage{multirow}

\DeclareFontEncoding{OT2}{}{} 
\newcommand{\textcyr}[1]{%
 {\fontencoding{OT2}\fontfamily{wncyr}\fontseries{m}\fontshape{n}\selectfont #1}}

\usepackage[all]{xy}

\newcommand{\Sha}{{\mbox{\textcyr{Sh}}}}

\newcommand{\bK}{{\bar{K}}}
\newcommand{\bF}{{\bar{F}}}
\newcommand{\bH}{{\bar{H}}}

\newcommand{\bC}{{\bar{C}}}
\newcommand{\bD}{{\bar{D}}}
\newcommand{\bA}{{\bar{A}}}

\newcommand{\x}{{[\mathbf{x}]}}

\newcommand{\y}{{[\mathbf{y}]}}
\newcommand{\Z}{{\mathbb Z}}
\newcommand{\Q}{{\mathbb Q}}

\newcommand{\F}{{\mathbb F}}

\newcommand{\A}{{\mathbb A}}

\newcommand{\PP}{{\mathbb P}}

\newcommand{\To}{\longrightarrow}
\newcommand{\Too}{\;\longrightarrow\;}

\newcommand{\magma}{{\tt Magma }}

\newcommand{\ord}{\operatorname{ord}}

\newcommand{\Jac}{\operatorname{Jac}}

\newcommand{\GL}{\operatorname{GL}}

\newcommand{\AGL}{\operatorname{AGL}}
\newcommand{\HH}{\operatorname{H}}
\newcommand{\PGL}{\operatorname{PGL}}

\newcommand{\Aut}{\operatorname{Aut}}

\newcommand{\Hom}{\operatorname{Hom}}
\newcommand{\Mod}{\operatorname{\,mod}}
\newcommand{\Map}{\operatorname{Map}}

\newcommand{\Spec}{\operatorname{Spec}}

\newcommand{\Sel}{\operatorname{Sel}}
\newcommand{\Pic}{\operatorname{Pic}}
\newcommand{\eps}{\varepsilon}

\newcommand{\Cov}{\operatorname{Cov}}
\newcommand{\Cl}{\operatorname{Cl}}

\newcommand{\pr}{\operatorname{pr}}

\newcommand{\res}{\operatorname{res}}
\newcommand{\unr}{{\text{\rm unr}}}

\newcommand{\diw}{\operatorname{div}}
\newcommand{\Div}{\operatorname{Div}}
\newcommand{\Princ}{\operatorname{Princ}}

\newcommand{\Aff}{\operatorname{Aff}}

\newcommand{\Br}{\operatorname{Br}}

\newcommand{\Ob}{\operatorname{Ob}}

\newenvironment{Proof}{\par\noindent{\sc Proof:}}%
                      {\hspace*{\fill}\nobreak$\Box$\par\hspace{2mm}}
\newenvironment{Remark}{\par\noindent{\sc Remark:}}{\par\hspace{2mm}}
                      {\hspace*{\fill}\nobreak$\Box$\par}

   {\begin{list}{}{\settowidth{\labelwidth}{#1}%
                   \setlength{\leftmargin}{\labelwidth}%
                   \addtolength{\leftmargin}{\labelsep}}}%
   {\end{list}}

\newtheorem{Theorem}{Theorem}[section]
\newtheorem{Lemma}[Theorem]{Lemma}
\newtheorem{Proposition}[Theorem]{Proposition}
\newtheorem{Corollary}[Theorem]{Corollary}

\newtheorem{Definition}[Theorem]{Definition}

\numberwithin{equation}{section}

\title[Second $p$-descent]{Second $p$-descents on elliptic curves}

\author{Brendan Creutz}
\address{School of Mathematics and Statistics, Carslaw Building F07, University of Sydney, NSW 2006, Australia}
\email{brendan.creutz@sydney.edu.au}
\date{20 April 2012}

\begin{document}

\begin{abstract}
Let $p$ be a prime and let $C$ be a genus one curve over a number field $k$ representing an element of order dividing $p$ in the Shafarevich-Tate group of its Jacobian. We describe an algorithm which computes the set of $D$ in the Shafarevich-Tate group such that $pD = C$ and obtains explicit models for these $D$ as curves in projective space. This leads to a practical algorithm for performing explicit $9$-descents on elliptic curves over $\Q$.
\end{abstract}

\maketitle

\section{Introduction}
Let $E$ be an elliptic curve over a number field $k$. The celebrated Mordell-Weil Theorem asserts that the set $E(k)$ of $k$-rational points on $E$ is a finitely generated abelian group. One would like to be able to determine this group in practice. In addition to the Mordell-Weil group, there is another important arithmetic invariant of an elliptic curve: the Shafarevich-Tate group $\Sha(E/k)$.  An $n$-descent on an elliptic curve is a way to obtain information on both of these groups.  For each integer $n\ge 2$, there is an exact sequence of finite abelian groups relating the two:
\[0\rightarrow E(k)/nE(k)\rightarrow \Sel^{(n)}(E/k)\rightarrow\Sha(E/k)[n]\rightarrow 0.\]
The middle term is a finite group known as the $n$-Selmer group. An {\em explicit $n$-descent} on $E$ computes the $n$-Selmer group and represents its elements as curves in projective space. Determination of $\Sel^{(n)}(E/k)$ yields explicit bounds on the Mordell-Weil rank and partial information on the Shafarevich-Tate group. In addition, the models produced can often be used to find generators of large height in the Mordell-Weil group or to study explicit counterexamples to the Hasse principle.

The technique of descent to study solutions of Diophantine equations goes back at least to Fermat. The idea is to parameterize the rational points on a given variety by the rational points on some finite collection of coverings. Reichardt and Lind used descent arguments to produce the first known counterexample to the Hasse principle \cite{Lind, Reichardt}. In a similar spirit Selmer studied diagonal cubic curves to systematically obtain counterexamples of degree $3$ \cite{Selmer}. The algorithm presented in this paper generalizes his method to arbitrary cubic curves. In one of the first applications of computers to number theory Birch and Swinnerton-Dyer \cite{BSD} studied the Mordell-Weil groups of elliptic curves over $\Q$ using $2$-descents. These computations produced empirical evidence motivating their famous conjecture. Together with deep results of Kolyvagin, Wiles and others \cite{modularity2, Kolyvagin, modularity1}, descents have now been used to verify the full conjecture for all elliptic curves over $\Q$ of rank $\le 1$ and conductor $\le 5000$ \cite{CreutzMiller}.

Historically $2$-descents were first performed using an explicit enumeration of certain homogeneous spaces of the elliptic curve. While applicable in principle to larger $n$ and over arbitrary number fields, the method is quickly defeated by combinatorial explosion when one ventures much beyond $2$-descents over $\Q$. There is an alternative approach which is based more closely on the original proof of the Mordell-Weil theorem \cite{Mordell,WeilDescent}. First one computes the $n$-Selmer group as a subgroup of a finite exponent quotient of the multiplicative group of some \'etale $k$-algebra. One is then left with the task of constructing explicit models for the coverings from representatives in the algebra. The first step requires deep arithmetic knowledge, such as $S$-class and -unit group information, of the constituent fields of the algebra.  It is only in the past two decades or so that improved computing power, higher-level computer algebra software and better theoretical understanding have made computations using this alternative approach feasible. 

Typically the algebra is related in some way to the group $E[n]$ of $n$-torsion points on $E$. For arbitrary $n$ there is an algorithm involving the \'etale algebra $R = \Map_k(E[n]\times E[n],\bar{k})$ of Galois equivariant maps from $E[n]\times E[n]$ to an algebraic closure of $k$ (see \cite[I.3.2]{CFOSS}). Generically $R$ contains an extension of $k$ of degree $O(n^4)$ making the arithmetic computations infeasible in practice. In general one can reduce computation of the $n$-Selmer group to the case that $n$ is a prime power. For $n=p$ a prime, there is a method using the \'etale $k$-algebra $\Map_k(E[p],\bar{k})$. Generically this splits as a product of $k$ with some field extension $A$ of degree $p^2-1$. The $p$-Selmer group is then computed as a subgroup of $A^\times/A^{\times p}$. The situation for $n=2$ is described in \cite{Simon1} and in \cite{SchJAC,StollIMP} where $2$-descent on Jacobians of hyperelliptic curves is also considered. For odd $p$, this was developed in the papers \cite{DSS,SchaeferStoll}. For $p=2,3$, these algorithms are practical over number fields of moderate degree and discriminant and are part of the \magma computer algebra package \cite{magma}. For larger $p$ these computations may be possible when $E$ admits a $p$-isogeny \cite{CreutzMiller,FisherThesis,FisherJEMS,MillerStoll,SchaeferStoll}, but for general elliptic curves $p$-descents for $p \ge 5$ are rather impractical.

In the second step, one starts with representatives for the $n$-Selmer group in some \'etale algebra and wants to construct explicit models for the corresponding coverings. For $n\ge 3$, the problem is studied in the series of papers \cite{CFOSS}, the situation for $n=2$ having been well-known for some time \cite[Section 15]{CaLectures}. For $n = 2$, one gets double covers of the projective line ramified in $4$ points. The models for $n = 3$ are plane cubics. For $n \ge 4$ one has degree $n$ curves in $\PP^{n-1}$ whose homogeneous ideal is generated by $n(n-3)/2$ quadrics. Examples of the utility of such models for computing generators of large height abound: \cite{BC,MinRed,Fisher6and12} all contain striking examples. In addition to this and the ability to produce explicit elements in the Shafarevich-Tate group, obtaining such models explicitly opens the possibility of doing higher descents.

To our knowledge, the only previously existing practical methods for computing the $n$-Selmer group of a general elliptic curve when $n $ is a higher prime power are for $n=4$ \cite{BruinStoll, Ca4d,MSS,Womack} and $n=8$ \cite{Stamminger}. Rather than performing a direct $2^m$-descent, these proceed by performing $2$-descents in a tower of coverings. For example, the output of a $2$-descent on $E$ is a finite collection of double covers of $\PP^1$ ramified in four points. A second $2$-descent computes the collection of everywhere locally solvable $2$-coverings of one (or all) of these. The running time is dominated by the computation of arithmetic information in the \'etale algebra corresponding to the ramification points of the double cover of $\PP^1$. This is typically a field of degree $4$, making the algorithm far more efficient than a direct $4$-descent on $E$. The output of a second $2$-descent is a finite collection of quadric intersections in $\PP^3$. These become the input for Stamminger's method for third $2$-descent. 

We present the analogous method for performing an explicit $p^2$-descent on an elliptic curve when $p$ is an odd prime. The first step is an explicit $p$-descent on $E$, yielding models for the elements of $\Sel^{(p)}(E/k)$ as genus one normal curves of degree $p$ in $\PP^{p-1}$. We then perform an explicit $p$-descent on some curve $C$ thus obtained. This is a computation of the finite set $\Sel^{(p)}(C/k)$ of everywhere locally solvable $p$-coverings of $C$ which produces explicit models for its elements as genus one normal curves of degree $p^2$ in $\PP^{p^2-1}$. Performing this computation for each element in $\Sel^{(p)}(E/k)$ one obtains information that is just as good as that obtained by an explicit $p^2$-descent on $E$. 

As is typical for descents, the running time of our algorithm is dominated by the computation of class and unit group information in a certain \'etale $k$-algebra. In our situation this is the \'etale $k$-algebra of degree $p^2$ corresponding to the set of flex points of $C$. In addition, our algorithm requires computations in a second \'etale algebra of degree at most $p^2(p^2-1)/2$. The most expensive operation required there is the extraction of $p$-th roots of elements known to be $p$-th powers. When $p=3$, one can get away with an algebra of degree $12$, where such computations are entirely feasible. For larger $p$, however, this provides another barrier to the practical applicability of the algorithm.

Our algorithm is practical for $p=3$ and $k = \Q$ and we have implemented it in the computer algebra system \magma\cite{magma}. Using this we are able to exhibit elements of order $9$ in the Shafarevich-Tate group of an elliptic curve. Alternatively, if the $3$-primary part of the Shafarevich-Tate group has exponent $3$, then the algorithm can be used to prove this unconditionally. We give several examples. In the first we consider the smallest elliptic curve (ordered by conductor) with elements of order $9$ in its Shafarevich-Tate group. Using first and second $3$-descents on on it and an isogenous curve we compute the full $3$-primary part of $\Sha$. In a second example, we give an explicit model in $\PP^8$ for an element of order $9$ in the Shafarevich-Tate group of an elliptic curve with irreducible mod $3$ representation. We also give an example verifying the full Birch and Swinnerton-Dyer conjecture for a CM elliptic curve over $\Q$ with Shafarevich-Tate group of order $144$.

\subsection{Organization}

The first two sections contain the necessary background information for setting up our descent. The majority of the material here should be well known to the experts. The only possible exceptions are our extension of the obstruction map to composite coverings and Proposition \ref{DescentOnPic}, which is slightly more general than the the usual results appearing in the literature. In Section \ref{DerivedGK} we establish a rather general framework for studying Picard groups of curves by using Galois equivariant families of functions on the curve. Ideas of this ilk have tended to play a role in most methods of explicit descent. Section \ref{ncoverings} introduces the main objects of study: $n$-coverings. We deal with their basic properties and those of related objects such as genus one normal curves and the obstruction map (most of this can be found in \cite[I]{CFOSS}). 

The following three sections develop the theoretical basis for the algorithm and the cohomological interpretation of second $p$-descents. For the most part we work with a fixed genus one normal curve $C$ of degree $p$ defined over an arbitrary perfect field of characteristic not equal to $p$. The main object of study is the set of isomorphism classes of $p$-coverings of $C$ with trivial obstruction. The primary tool is the {\em descent map}, which gives a concrete algebraic realization of this rather abstractly defined set.

In Section \ref{AffineStructure} we identify the domain of the descent map as a principal homogeneous space for a certain subgroup of $\HH^1(K,E[p])$. We give both a cohomological description of this subgroup and explicit representations of its members as elements of the multiplicative group of a certain \'etale $K$-algebra $H$. The descent map is then defined in Section \ref{TheDescentMap} as a map taking values in a certain quotient of $H^\times$. Ultimately we will see that the descent map may be interpreted as an affine map (loosely speaking a linear map followed by a translation), so that the material of Section \ref{AffineStructure} can be understood as a study of its linear part. In Sections \ref{InjectivityOfTheDescentMap} and \ref{ImageOfTheDescentMap} we show that the descent map is injective and determine its image. Then in Section \ref{InverseOfTheDescentMap} we construct an explicit inverse to the descent map, which shows how to obtain explicit models for elements in the image of the descent map as genus one normal curves of degree $p^2$ in $\PP^{p^2-1}$.

Following this we specialize in Section \ref{ComputingTheSelmerSet} to the case that the base field is a number field. Armed with the material of the preceding sections, computation of the Selmer set is almost routine. The descent map gives a bijection from $\Sel^{(p)}(C/k)$ onto its image, which we call the algebraic $p$-Selmer set. This gives an algebraic presentation of the $p$-Selmer set which is amenable to machine computation. After outlining the complete algorithm we conclude with a small selection of examples in Section \ref{Examples}.

\subsection{Notation}
If $\mathcal{G}$ is an abelian group and $n \ge 1$, we use $\mathcal{G}[n]$ to denote the subgroup of $\mathcal{G}$ consisting of elements which are killed by $n$. For a commutative ring $R$ we use $R^\times$ to denote the multiplicative group of invertible elements.

The symbol $K$ will always denote a perfect field and $p$ will always denote a prime number not equal to the characteristic of $K$. We always assume we have a fixed algebraic closure $\bK$ of $K$ and any algebraic extension of $K$ we write down is taken to be a subfield of $\bK$. We write $G_K$ for the absolute Galois group of $K$. For $n$ prime to the characteristic of $K$ we use $\mu_n$ to denote the $G_K$-module of $n$-th roots of unity in $\bK$. Since we restrict to perfect fields, the term local field will be used to mean the completion of a number field at some prime.

If $K \subset L$ is an extension of fields and $A$ is a $K$-algebra, then $A\otimes_KL$ is an $L$-algebra which we will denote simply by $A_L$. In the particular case $L = \bK$, the notation $\bA$ will also be used to denote $A\otimes_K\bK$. If $\phi : A \to B$ is a morphism of $K$-algebras, the induced map $A_L \to B_L$ will also be denoted by $\phi$. 

For a smooth, projective and absolutely irreducible curve $C$ defined over $K$ and a commutative $K$-algebra $A$, we write $C\otimes_KA$ for the scheme $C \times_{\Spec(K)}\Spec(A)$. When $A = \bK$, we also write $\bC = C \otimes_K \bK$. We use $\kappa(\bC)$ and $\kappa(C)$ to denote the function field of $\bC$ and its $G_K$-invariant subfield, respectively. We use $\Div(\bC)$ to denote the free abelian group on the set of $\bK$-points of $C$. Its elements are called divisors and will often be written as integral linear combinations of points. If we wish to make it clear that we are considering a point as a divisor, we will use square brackets. So $[P] \in \Div(\bC)$ is the divisor corresponding to $P \in C(\bK)$. The action of $G_K$ on points extends to an action on divisors. We use $\Div(C)$ for the $G_K$-invariant subgroup and refer to its elements as $K$-rational divisors. A closed point of the $K$-scheme $C$ corresponds to a Galois orbit of points in $C(\bK)$. As such a closed point may be interpreted as an element of $\Div(C)$. In fact, $\Div(C)$ is the free abelian group on such closed points. We denote the divisor of a function $f \in \kappa(\bC)^\times$ by $\diw(f)$.

Two divisors are said to be linearly equivalent if their difference is equal to $\diw(f)$ for some rational function $f\in\kappa(\bC)^\times$. The group of principal divisors is $\Princ(\bC) = \{ \diw(f) : f\in\kappa(\bC)^\times \}$. It follows from Hilbert's Theorem 90 that the $G_K$-invariant subgroup, $\Princ(C)= \Princ(\bC)^{G_K}$, is the group of divisors that are divisors of functions in $\kappa(C)^\times$. We use $\Pic(\bC)$ to denote the group of divisors modulo principal divisors. We remind the reader that not every $K$-rational divisor class can be represented by a $K$-rational divisor (for an example see \cite{CaV}). Since the degree of the divisor associated to a rational function is $0$, there is a well-defined notion of degree for divisor classes in $\Pic(\bC)$. We denote the set of divisor classes of degree $i \in \Z$ by $\Pic^i(\bC)$. We use similar notation for the other groups defined above. The group $\Pic^0(\bC)^{G_K}$ may identified with the group of $K$-rational points on the Jacobian of $C$.

\subsection*{Acknowledgements} The majority of this paper is taken from the author's PhD thesis \cite{CreutzThesis}. It is a pleasure to thank Michael Stoll for introducing me to this topic and for sharing his insights, Tom Fisher for his careful reading of the thesis and several useful comments, Steve Donnelly for helpful discussions relating to the implementation and the anonymous referee for their remarks.

\section{Derived $G_K$-sets and descent on Picard groups}
\label{DerivedGK}
In this section we give a rather general recipe for writing down homomorphisms from the Picard group of a curve into a finite exponent quotient of some \'etale $K$-algebra. The vast majority of explicit descents, whether they be on curves or their Jacobians, make use of such a map (see for example the philosophy described in \cite{SchJAC}).

\subsection{Etale $K$-algebras and $G_K$-sets}
\label{GKEtale}
If $\Omega$ is a finite $G_K$-set, define $\bA(\Omega) = \Map(\Omega,\bK)$ to be the $\bK$-algebra of maps from $\Omega$ to $\bK$. There is a natural action of $G_K$ on $\bA(\Omega)$ defined by
\[ \phi^\sigma : x \mapsto \left(\phi(x^{\sigma^{-1}})\right)^\sigma\,.\] As a $\bK$-algebra $\bA(\Omega)$ is isomorphic to $\prod_{i=1}^{\#\Omega}\bK$, but the action of $G_K$ is twisted by the action on $\Omega$. The $G_K$ invariant subspace of $\bA(\Omega)$ is the space of $G_K$-equivariant maps $\Map_K(\Omega,\bK) := \Map(\Omega,\bK)^{G_K}$. This is an \'etale $K$-algebra; it splits as a product of finite extensions of $K$ corresponding to the orbits in $\Omega$. This defines an anti-equivalence between the categories of finite $G_K$-sets and \'etale $K$-algebras.

We will frequently find ourselves working with objects defined over such algebras, e.g. varieties, points, functions, etc. From a scheme-theoretic point of view this presents no difficulty. It will, however, be convenient to interpret these objects as Galois equivariant maps. For example suppose $C$ is a $K$-variety and $A = \Map_K(\Omega,\bK)$.  Then $C\otimes_K A$ is a scheme over $A$. Geometrically it is a disjoint union of copies of $\bar{C}$ parameterized by $\Omega$. We will abuse notation by writing $\kappa(C\otimes_K A)$ for $\kappa(C) \otimes_K A$ and referring to its elements as rational functions on $C\otimes_K A$. We may interpret such a rational function as a Galois equivariant map $\Omega \to \kappa(\bC)^\times$ or equivalently as a Galois equivariant family of rational functions $f_\omega \in \kappa(\bC)^\times$ indexed by $\omega \in \Omega$. The Galois equivariance means that $(f_\omega)^\sigma = f_{\omega^\sigma}$ for all $\sigma \in G_K$. The divisor of $f$ can be interpreted as a Galois equivariant map $\Omega \to \Div(\bC)$, and so on.

Similarly, for $n$ prime to the characteristic of $K$, we define $\mu_n(\bar{A}) = \Map(\Omega,\mu_n)$. This is the $n$-torsion subgroup of $\bar{A}^\times$. We mention here the generalization of Hilbert's Theorem 90 to \'etale algebras which states that $\HH^1(K,{\bA}^\times) = 0$. To prove it one uses Shapiro's Lemma to reduce to the usual Theorem 90 for fields. Using this with the Kummer sequence (as one does for fields) one is lead to an isomorphism $\HH^1(K,\mu_n(\bA)) \simeq A^\times/A^{\times n}$.

If $\Psi,\Omega$ are finite $G_K$-sets, we will say that $\Psi$ is {\em derived from} $\Omega$ if the elements of $\Psi$ are unordered tuples (multisets) of elements of $\Omega$ and the action on $\Psi$ is induced by that on $\Omega$. For example, the set of unordered pairs (distinct or not) of elements in $\Omega$ is a derived $G_K$-set. One can also interpret the elements of $\Psi$ as formal integral linear combinations of elements of $\Omega$ with nonnegative coefficients. In this interpretation we write $b \in \Psi$ as a sum $\sum n_a a$ of distinct elements $a \in \Omega$.

To $\Omega$ and $\Psi$ we associate \'etale $K$-algebras, $A(\Omega) = \Map_K(\Omega,\bK)$ and $A(\Psi) = \Map_K(\Psi,\bK)$. If $\Psi$ is derived from $\Omega$ as a $G_K$-set, then we define the {\em induced norm maps} between the corresponding algebras:
\[ A(\Omega) = \Map_K(\Omega,\bK) \ni \phi \mapsto \left( (b=\sum n_aa) \mapsto \prod_{a}\phi(a)^{n_a} \right) \in \Map_K(\Psi,\bK) = A(\Psi)\,. \]

\subsection{Descent on Picard groups}
Let $C$ be a smooth, absolutely irreducible projective curve over a perfect field $K$. Let $\Omega \subset C(\bK)$ be a finite $G_K$-set of geometric points on $C$ and $\Psi \subset \Div(\bar{C})$ a finite $G_K$-set of effective divisors on $C$ which are supported on $\Omega$. As above $A(\Omega)$ and $A(\Psi)$ denote the corresponding \'etale $K$-algebras. As a $G_K$-set $\Psi$ is derived from $\Omega$ and we have an induced norm map $\partial:A(\Omega) \to A(\Psi)$.

Now consider a rational function 
\[ f = (f_\psi) \in \kappa(C \otimes_K A(\Psi)\,)^\times = \Map_K(\Psi,\kappa(\bar{C})^\times)\,.\] We consider this either as a function on $C \otimes_K A(\Psi)$ or as a Galois equivariant family of rational functions in $\kappa(\bar{C})^\times$ parameterized by $\psi \in \Psi$. We interpret the divisor of $f$, an element of $\Div(C\otimes_K A(\Psi))$, as a $G_K$-equivariant map $\Psi \to \Div(\bar{C})$. We write $\diw(f)$ as a difference of effective divisors. This can be interpreted as a difference of a pair of $G_K$-equivariant maps $[f]_0, [f]_\infty : \Psi \to \Div(\bar{C})$ whose values at $\psi \in \Psi$ are the zero and pole divisors of $f_\psi$, respectively. Now suppose $f$ satisfies
\begin{enumerate}
\item $\forall \psi \in \Psi$, $[f]_0(\psi) = \psi$, and 
\item $\forall \psi \in \Psi$ and $\sigma \in G_K$, $[f]_\infty(\psi^\sigma) = [f]_\infty(\psi)$.
\end{enumerate}
The first condition says that $\psi$ is the zero divisor of the function $f_\psi \in \kappa(\bC)^\times$. The second condition amounts to saying that the map $[f]_\infty:\Psi\to\Div(\bC)$ is constant on each $G_K$-orbit in $\Psi$. The Galois equivariance then implies that, on each orbit $\mathcal{O} \subset \Psi$, the value of $[f]_\infty$ is some $K$-rational divisor $d_\mathcal{O} \in \Div(C)$. We write each as a sum
\[ d_\mathcal{O} = \sum_{\mathcal{P}} m_{\mathcal{O},\mathcal{P}}\mathcal{P}\,,\] of closed points $\mathcal{P}$ and set $m_\mathcal{O} = \gcd(m_{\mathcal{O},\mathcal{P}})$ (recall that a closed point corresponds to a $G_K$-orbit of points in $C(\bK)$). Using these weights, we define a map \[\iota: K \ni a \mapsto (a^{m_\mathcal{O}})_{\mathcal{O}} \in \prod_{\mathcal{O}} A(\mathcal{O}) = A(\Psi) \]

\begin{Proposition}
\label{DescentOnPic}
With notation as above, let $d = \sum_Pn_P[P] \in \Div(C)$ (written as a sum of $\bK$-points of $C$) be any $K$-rational divisor on $C$ with support disjoint from all zeros and poles of the $f_\psi$.
\begin{enumerate}
\item Evaluating $f$ on $d$ gives a well-defined element $f(d) := \prod_{P} f(P)^{n_P} \in A(\Psi)^\times$.
\item $f$ induces a unique homomorphism 
\[ \Pic(C) \to \frac{A(\Psi)^\times}{\iota(K^\times)\partial(A(\Omega)^\times)}\] with the property that, for all $d$ as above,  the image of the class of $d$ is equal to the class of $f(d)$.
\end{enumerate}
\end{Proposition}

\begin{Proof}
Any rational function $h \in \kappa(\bC)^\times$ defines a homomorphism from the group of divisors of $C$ with support disjoint from the support of $\diw(h)$ to the multiplicative group of $\bK$ by 
\[ d = \sum n_P P \mapsto h(d) = \prod h(P)^{n_P} \in \bK^\times\,.\] If $K'$ is some extension of $K$ and $h$ is defined over $K'$, then this restricts to give a homomorphism from the group of $K'$-rational divisors with support disjoint from that of $\diw(h)$ into $K'^\times$. If $f$ is as in the proposition, then it is defined over $A(\Psi)$ which splits as a product of extensions of $K$, so the first statement in the proposition is clear.

For the second, define \[\phi_f :\Pic(C) \to \frac{A(\Psi)^\times}{\iota(K^\times)\partial(A(\Omega)^\times)}\] by setting the value of $\phi_f$ on $\Xi \in \Pic(C)$ equal to the class of $f(d)$, where $d \in \Div(C)$ is any $K$-rational divisor representing $\Xi$ with support disjoint from $\Omega$ and $[f]_\infty$. If this is well-defined, then it is clearly the unique homomorphism with the stated property.

First we argue that such $d$ exists. This follows from \cite[page 166]{LangAV} where it is shown that any $K$-rational divisor class which is represented by a $K$-rational divisor contains a $K$-rational divisor avoiding a given finite set (see also \cite[footnote to page 4]{SikShort}). Next we use Weil reciprocity to show that the result does not depend on the choice for $d$.

Let $h \in \kappa(C)^\times$ be any rational function whose zeros and poles are disjoint from those of all of the $f_\psi$. We will show that $f(\diw(h)) \in \iota(K^\times)\partial(A(\Omega)^\times)$, from which the proposition follows. For each $\psi \in \Psi$, the divisor of $h$ is prime to \[\diw(f_\psi) = [f]_0(\psi) - [f]_\infty(\psi) = \psi - [f]_\infty(\psi)\,.\] So by Weil reciprocity, \[f_\psi(\diw(h)) = h( \diw(f_\psi)) = \frac{ h([f]_0(\psi))}{h([f]_\infty(\psi))}\,.\] Interpreting this as a map we have \[ f(\diw(h)) = \frac{h([f]_0)}{h([f]_\infty)} \in \Map_K(\Psi,\bK^\times) = A(\Psi)^\times\,.\] Define $\alpha \in \Map_K(\Omega,\bK^\times) = A(\Omega)^\times$ by $\alpha:\Omega \ni \omega \mapsto h(\omega) \in \bK^\times$. Now consider $\partial(\alpha) \in \Map_K(\Psi,\bK^\times) = A(\Psi)^\times$. The value of $\partial(\alpha)$ at $\psi = \sum n_\omega \omega \in \Psi$ is \[ \partial(\alpha)_\psi = \prod \alpha(\omega)^{n_\omega} = \prod h(\omega)^{n_\omega} = h(\psi) = h([f]_0(\psi))\,.\] This shows that $h([f]_0) = \partial(\alpha) \in \partial(A(\Omega)^\times)$.

It remains to show that $h([f]_\infty) \in \iota(K^\times)$. Recall that the value of  $[f]_\infty$ on the orbit $\mathcal{O} \subset \Psi$ is the divisor $d_{\mathcal{O}} = \sum_{\mathcal{P}}m_{\mathcal{O},\mathcal{P}}\mathcal{P}$ and that $m_{\mathcal{O}} = \gcd(m_{\mathcal{O},\mathcal{P}})$. The $\mathcal{P}$ are closed points on $C$. In particular, each is a $K$-rational divisor and we know how to evaluate $h$ at $\mathcal{P}$ to obtain an element in $K^\times$. Extending by linearity we have that the value taken by $h([f]_\infty)$ on any $G_K$-orbit $\mathcal{O} \subset \Psi$ is $\prod_{\mathcal{P}} h(\mathcal{P})^{m_{\mathcal{O},\mathcal{P}}}$, which is a product of $m_{\mathcal{O},\mathcal{P}}$-th powers in $K^\times$. A product of $m_{\mathcal{O},\mathcal{P}}$-th powers is clearly a $\gcd(m_{\mathcal{O},\mathcal{P}})$-th power, so the value of $h([f]_\infty)$ on $\mathcal{O}$ is in $K^{\times m_{\mathcal{O}}}$. It follows that $h([f]_\infty) \in \iota(K^\times)$. This completes the proof.
\end{Proof}

\begin{Remark}
To do a $p$-descent on an elliptic curve $E$ with Weierstrass equation $y^2 = f(x)$ one typically uses a Galois equivariant family of functions with zeros of order $p$ at the nontrivial $p$-torsion points and poles of order $p$ at the identity. For example, when $p = 2$ one can take the family of functions $x - \theta$ where $\theta$ runs through the roots of $f(x)$. The Proposition gives a homomorphism $\Pic(E) \to A^\times/A^{\times p}$, where $A$ is the \'etale algebra associated to the $G_K$-set of nontrivial $p$-torsion points. The restriction of this to $\Pic^0(E) = E(K)$ can be identified with connecting homomorphism in the Kummer sequence for $E$ associated to multiplication by $p$.
\end{Remark}

\section{$n$-coverings}
\label{ncoverings}

\begin{Definition}
Let $C$ be a smooth, projective and absolutely irreducible curve defined over $K$ with Jacobian $E$ and $n \ge 2$ prime to the characteristic of $K$. An {\em $n$-covering of $C$} is an \'etale morphism $\pi:D\to C$ of absolutely irreducible curves which is geometrically Galois with group isomorphic to $E[n]$ as a $G_K$-module. Two $n$-coverings are isomorphic if they are isomorphic as $C$-schemes. We denote the set of all $K$-isomorphism classes of $n$-coverings of $C$ defined over $K$ by $\Cov^{(n)}(C/K)$. When $K = k$ is a number field, we define the {\em $n$-Selmer set of $C$}, $\Sel^{(n)}(C/k)$, to be the set of $k$-isomorphism classes of $n$-coverings of $C$ which have points everywhere locally.
\end{Definition}

Geometrically, every $n$-covering of $C$ is obtained by pulling-back the multiplication by $n$ map via a suitable embedding of $C$ into its Jacobian. This follows from geometric class field theory and the fact that there is a unique subgroup of $E(\bar{K})$ isomorphic to $E[n]$. When $C=E$ is an elliptic curve, every $n$-covering of $E$ can be viewed as a twist of the multiplication by $n$ map on $E$. This gives a canonical choice for the trivial $n$-covering of $E$. So, by the twisting principle, $\Cov^{(n)}(E/K)$ is canonically isomorphic to $\HH^1(K,E[n])$ and, as such, carries the structure of an abelian group. More generally, all $n$-coverings of $C$ are twists of one another. So (provided it is nonempty) we may consider $\Cov^{(n)}(C/K)$ as principal homogeneous space for $\HH^1(K,E[n])$ with the action being given by twisting.

The $n$-Selmer set is finite and (at least in principle) computable \cite{ChevalleyWeil}. We refer to its computation as an {\em $n$-descent on $C$}. When $C =E$ is an elliptic curve the $n$-Selmer set is a finite subgroup of $\HH^1(k,E[n])$ and sits in a short exact sequence
\begin{align}
\label{SelSeq}
0 \to E(k)/nE(k) \to \Sel^{(n)}(E/k) \to \Sha(k,E)[n] \to 0\,
\end{align}
(see \cite[Theorem X.4.2]{Silverman}). Classical descent theory (going back to Chevalley-Weil \cite{ChevalleyWeil}) tells us that the $n$-Selmer set yields a finite partition of the rational points on $C$: \[ C(k) = \bigcup_{(D,\pi) \in \Sel^{(n)}(C/k)} \pi(D(k))\,.\] In particular, the emptiness of the $n$-Selmer set is an obstruction to the existence of rational points on $C$. But even when $C(k) \ne \emptyset$, explicitly computing the $n$-Selmer set can be very useful for finding points of large height on $C$ (and ultimately generators of large height on the Jacobian of $C$).

\subsection{Projective models}
Let $C$ be a smooth, projective and absolutely irreducible genus one curve defined over $K$ with Jacobian $E$. For any $K$-rational divisor of degree $n \ge 2$ on $C$, the associated complete linear system gives rise to a morphism from $C$ to $\PP^{n-1}$ defined over $K$. For $n = 2$ this results in a double cover of $\PP^1$ ramified in four points. For $n \ge 3$, the complete linear system yields an embedding $C \hookrightarrow \PP^{n-1}$. The image is called a {\em genus one normal curve of degree $n$}. For $n=3$, the image of $C$ is a plane cubic curve. For larger $n$, the homogeneous ideal of the image is generated by a $K$-vector space of quadrics of dimension $n(n-3)/2$. Two genus one normal curves of degree $n$ are said to be $K$-equivalent if they can be identified by a $K$-automorphism of $\PP^{n-1}$.

More generally, the complete linear system associated to any $K$-rational divisor class of degree $n \ge 2$ on $C$  gives rise to a $K$-morphism $C \to S$ from $C$ to a Brauer-Severi variety $S$ of dimension $n-1$ (see \cite[p. 160]{SerreLF}, \cite[I.1.20]{CFOSS} or \cite[Section 3]{ClarkWC1}). Conversely, any morphism from $C$ to an $n-1$ dimensional Brauer-Severi variety $S$ such that the image is not contained in a linear subvariety gives rise to a $K$-rational divisor class on $C$ by pulling back the class of a hyperplane on $\bar{S} \simeq \PP^{n-1}$.

This leads to the notions of {\em torsor divisor class pairs} and {\em Brauer-Severi diagrams}. The data for a torsor divisor class pair consists of a $K$-torsor $C$ under $E$ and a $K$-rational divisor class of degree $n$ on $C$. The corresponding morphism $C \to S$ is called a Brauer-Severi diagram. In \cite[I, Section 1]{CFOSS} it is shown that, up to appropriate notions of isomorphism, torsor divisor class pairs and Brauer-Severi diagrams are both parameterized by the group $\HH^1(K,E[n])$. Recall that this group also parameterizes $n$-coverings of $E$. If $(C,\rho)$ is an $n$-covering, then there exists an isomorphism $\psi:C\to E$ defined over $\bK$ such that $\rho = n\circ\psi$. This gives $C$ the structure of a $K$-torsor under $E$. The pull-back of $n[0_E]$ by $\psi$ defines a $K$-rational divisor class on $C$. One can show that this gives a torsor divisor class pair, whose class in $\HH^1(K,E[n])$ is the same as that of $(C,\rho)$. Thus the Brauer-Severi diagram corresponding to $(C,\rho)$ is the map $C \to S$ given by the complete linear system associated to the divisor $\psi^*n[0_E]$. This results in a model for $C$ as a genus one normal curve of degree $n$ in $\PP^{n-1}$ if and only if $\psi^*n[0_E]$ is linearly equivalent to some $K$-rational divisor.

Conversely, a genus one normal curve $C$ of degree $n$ determines the class of an $n$-covering $(C,\rho)$ in $\Cov^{(n)}(E/K) = \HH^1(K,E[n])$ up to composition with an automorphism of $E$. The possibilities correspond to the finitely many nonisomorphic structures for $C$ as a $K$-torsor under $E$.

\begin{Definition}
\label{FlexPoints}
A point $x$ on a genus one normal curve of degree $n \ge 3$ is called a {\em flex point} if there is a hyperplane in $\PP^{n}$ meeting $C$ in $x$ with multiplicity $n$.
\end{Definition}

The set of flex points $X$ on $C$ is defined by a geometric property, so it is a $G_K$-set. If $C$ is endowed with the structure of an $n$-covering $\rho:C \to E$, then $X = \rho^{-1}(O_E)$ is the fiber above the identity on $E$, and the corresponding action of $E$ on $C$ restricts to a simply transitive action of $E[n]$ on $X$. This gives $X$ the structure of an $E[n]$-torsor; its class in $\HH^1(K,E[n])$ is the same as that of the $n$-covering $(C,\rho)$ (see \cite[I Section 1]{CFOSS}). It follows also that $\#X = n^2$.

\subsection{The obstruction map}
\label{obstructionmap}
Recall that $\Pic(C)$ is the quotient of the group of $K$-rational divisors on $C$ by the group of $K$-rational principal divisors, while $\Pic(\bar{C})^{G_K}$ is the group of $K$-rational divisor classes. It follows from Hilbert's Theorem 90 that the obvious map $\Pic(C) \to \Pic(\bar{C})^{G_K}$ is injective. In general, however, it is not surjective. To measure this failure one is naturally led to use Galois cohomology. The Picard group is defined by the exact sequence \[ 1 \to \bK^\times \to \kappa(\bar{C})^\times \to \Div(\bar{C}) \to \Pic(\bar{C})\to 0\,.\] Taking Galois invariants, this sequence may no longer be exact. One can deduce an exact sequence \[ 0 \to \Pic(C) \to \Pic(\bar{C})^{G_K} \stackrel{\delta_C}{\To} \Br(K)\,, \] where $\Br(K)$ denotes the Brauer group of $K$. The map $\delta_C$ gives the obstruction to a $K$-rational divisor class being defined by a $K$-rational divisor.

Following \cite{CFOSS} we define the obstruction map
\[ \Ob_n:\HH^1(K,E[n]) \to \Br(K) \] by $\Ob_n(\xi) = \delta_C(\Xi)$, where $(C,\Xi)$ is any torsor divisor class pair representing the class $\xi \in \HH^1(K,E[n])$. From this definition we obtain the fundamental property of the obstruction map that the Brauer-Severi diagram corresponding to an $n$-covering $(C,\rho)$ gives a model for $C$ as a genus one normal curve of degree $n$ in $\PP^{n-1}$ if and only if $\Ob_n((C,\rho))=0$. Conversely, any genus one normal curve $C \to \PP^{n-1}$ of degree $n$, together with a structure of torsor under its Jacobian $E$ determines a unique isomorphism class of $n$-coverings of $E$ with trivial obstruction.

Using a result of Zarhin \cite{Zarhin} O'Neil has shown \cite{O'Neil} that the obstruction map is a quadratic form. This means that, for any integer $a$, $\Ob_n(a\xi) = a^2\Ob_n(\xi)$ and that the pairing \[(\xi,\xi')\mapsto \Ob_n(\xi+\xi')-\Ob_n(\xi)-\Ob_n(\xi')\] is bilinear. The pairing is in fact the cup product associated to the Weil pairing on $E[n]$, i.e. the composition
\[ \cup_n:\HH^1(K,E[n])\times\HH^1(K,E[n]) \stackrel{\cup}{\To} \HH^2(K,E[n]\otimes E[n]) \stackrel{e_n}{\To} \HH^2(K,\mu_n) \simeq \Br(K)[n]\,.\]

Using the compatibility of the Weil pairings of levels $mn$ and $n$ (where $m,n \ge 2$ are integers not divisible by the characteristic of $K$) and the fact that the bilinear form associated to the obstruction map is the Weil pairing cup product one can prove that the following diagram commutes. For details see \cite[Proposition 6]{ClarkSharif}.
\begin{align}
\label{Obmn}
\xymatrix{ \HH^1(K,E[m]) \ar[r]^{i_*}\ar[d]^{\Ob_m}& \HH^1(K,E[mn]) \ar[r]^{m_*}\ar[d]^{\Ob_{mn}}& \HH^1(K,E[n]) \ar[d]^{\Ob_n} \\
	\Br(K)[m] \ar[r]^n & \Br(K)[mn] \ar[r]^{m}& \Br(K)[n] }
\end{align}

Let $C$ be a genus one normal curve of degree $n$ defined over $K$ with Jacobian $E$. We would like to extend the obstruction map to the set $\Cov^{(m)}(C/K)$. To do this, fix a map $\rho:C \to E$ making $C$ into an $n$-covering of $E$. If $(D,\pi) \in \Cov^{(m)}(C/K)$ is an $m$-covering of $C$, then composing the covering maps gives $D$ the structure of an $mn$-covering of $E$. This gives a map, depending on both $C$ and $\rho$, \[ \Psi_\rho:\Cov^{(m)}(C/K) \ni (D,\pi) \mapsto (D,\rho\circ\pi) \in \Cov^{(mn)}(E/K) = \HH^1(K,E[mn])\,.\] 
Composing this with the obstruction map yields a map \[ \Ob_{mn}:\Cov^{(m)}(C/K) \stackrel{\Psi_\rho}{\To} \HH^1(K,E[mn]) \stackrel{\Ob_{mn}}{\Too} \Br(K)\,,\] which we also denote by $\Ob_{mn}$. Since $\rho$ is uniquely determined up to an automorphism of $E$, the set in the following definition depends only on the equivalence class of $C$ as a genus one normal curve of degree $n$ (and not on the additional choice for the torsor structure).

\begin{Definition}
We say that an $m$-covering $\pi:D\to C$ has {\em trivial obstruction} if its image under $\Ob_{mn}$ is trivial. We use
\[ \Cov_0^{(m)}(C/K) := \{ (D,\pi) \in \Cov^{(m)}(C/K) : \Ob_{mn}((D,\pi)) = 0 \} \] to denote the set of isomorphism classes of $m$-coverings of $C$ with trivial obstruction.
\end{Definition}

Our initial interest in this subset of coverings is justified by the following lemma. This is essentially a result of Cassels \cite[Theorem 1.2]{CaIV}. 
\begin{Lemma}
Let $C$ be a genus one normal curve of degree $n$ over a number field $k$. Then for any $m \ge 2$, $\Sel^{(m)}(C/k)$ is contained in $\Cov_0^{(m)}(C/k)$.
\end{Lemma}

\begin{Proof}
If $D$ is a smooth, projective and absolutely irreducible curve over a field $K$ and $D(K)\ne \emptyset$, then $\Pic(D)= \Pic(\bD)^{G_K}$. From the exact sequence $\Pic(D)\to\Pic(\bD)^{G_K}\stackrel{\delta_D}\To\Br(K)$ it follows that an $m$-covering $\pi:D \to C$ has trivial obstruction if $D(K) \ne \emptyset$. If $K=k$ is a number field and $D$ is everywhere locally solvable, then this is the case everywhere locally. The local global principle for the Brauer group of $k$ then implies that $\Pic(D)=\Pic(\bD)^{G_k}$. In particular, the elements of the $m$-Selmer set of $C$ (resp. $m$-Selmer group if $C = E$) have trivial obstruction.
\end{Proof}

We now consider the case when $m = n$. When $\Cov^{(n)}(C/K)$ is nonempty, it is a principal homogeneous space for $\HH^1(k,E[n])$. The action of a class represented by a cocycle $\xi$ on a covering $D$ is given by twisting. We use $D_\xi$ to denote the twist of $D$ by $\xi$. Both $D$ and $\xi$ have canonical images in $\HH^1(k,E[n^2])$ and the action of twisting coincides with the group law there. Namely, the image of $D_\xi$ is the sum of the images of $D$ and $\xi$. The following lemma identifies how the obstruction changes under this action.

\begin{Lemma}
\label{Obisaffine}
For $D \in \Cov^{(n)}(C/K)$ and $\xi \in \HH^1(K,E[n])$ we have \[\Ob_{n^2}(D_\xi)= C \cup_n \xi + \Ob_{n^2}(D)\,,\] where $\cup_n$ denotes the cup product associated to the Weil pairing of level $n$.
\end{Lemma}

\begin{Proof}
For the proof, we identify $D$, $D_\xi$ and $\xi$ with their images in $\HH^1(K,E[n^2])$.
We know that $\Ob_{n^2}$ is quadratic and that the associated bilinear form is given by $\cup_{n^2}$. This means that \[ D \cup_{n^2} \xi = \Ob_{n^2}(D_\xi) - \Ob_{n^2}(D) - \Ob_{n^2}(\xi)\,.\]
The compatibility of the obstruction maps of different levels demonstrated by (\ref{Obmn}) shows that $\Ob_{n^2}(\xi)=0$. On the other hand, the Weil pairings of levels $n^2$ and $n$ satisfy the compatibility condition (see \cite[III.8]{Silverman}): \[ \text{for all } S\in E[n^2] \text{ and } T \in E[n]\,,\,\, e_{n^2}(S,T)=e_n(nS,T)\,.\] For the cup product on the left-hand side above this means
\[ D \cup_{n^2} \xi= (n_*D) \cup_n \xi = C \cup_n \xi\,, \] which completes the proof.
\end{Proof}

We use $C^\perp$ to denote the annihilator of $C$ with respect to $\cup_n$, i.e. 
\[ C^\perp = \{ \xi \in \HH^1(K,E[n])\,:\,C\cup_n \xi = 0\,\}\,. \]
\begin{Corollary}
\label{Covphs}
The set $\Cov_0^{(n)}(C/K)$ is either empty or is a principal homogeneous space for $C^\perp \subset \HH^1(K,E[n])$.
\end{Corollary}

\begin{Proof}
This is clear from the lemma and the fact that the action of $\HH^1(K,E[n])$ on $\Cov^{(n)}(C/K)$ is compatible with the group law in $\HH^1(K,E[n^2])$
\end{Proof}

The following lemma gives an alternative characterization of $\Cov_0^{(n)}(C/K)$ which will be fundamental to our construction of the descent map in Section \ref{TheDescentMap}.
\begin{Lemma}
\label{pbflex}
Let $C$ be a genus one normal curve of degree $n$ with set of flex points $X$ and let $(D,\pi) \in \Cov^{(n)}(C/K)$. Then $(D,\pi)$ has trivial obstruction if and only if there exists a model for $D$ as a genus one normal curve of degree $n^2$ in $\PP^{n^2-1}$ defined over $K$ with the property that the pull-back of any $x \in X$ by $\pi$ is a hyperplane section.
\end{Lemma}

\begin{Proof}
Fix isomorphisms $\psi_D:D \to E$ and $\psi_C:C \to E$ (defined over $\bK$) and a covering map $\rho:C \to E$ such that the diagram \[ \xymatrix{ D \ar[d]_{\psi_D}\ar[r]^\pi& C \ar[d]_{\psi_C} \ar[r]^\rho& E \ar@{=}[d]\\
E\ar[r]_{n}& E\ar[r]_{n}& E }\] commutes. Then $X = \rho^{-1}(0_E)$. By definition $(D,\pi)$ has trivial obstruction if and only if $\psi_D^*(n^2[0_E])$ is linearly equivalent to some $K$-rational divisor. On the other hand, $D$ admits a model as in the statement of the lemma if and only if $\pi^*[x]$ is linearly equivalent to some $K$-rational divisor, for each $x \in X$. It thus suffices to show, for all $x \in X$, that $\psi_D^*(n^2[0_E])$ and $\pi^*[x]$ are linearly equivalent. For this we may work geometrically. The problem is then equivalent to showing that for any $n$-torsion point $P \in E[n]$, the pull-back of $P$ under the multiplication by $n$ isogeny is linearly equivalent to $n^2[0_E]$. This follows from the well-known fact that two divisors on an elliptic curve are linearly equivalent if and only if they have the same degree and the same sum. Indeed, the divisors in question both have degree $n^2$ and sum to $0_E$ in the group $E(\bK)$.
\end{Proof}

\subsection{Composite coverings}
\label{compcov}
Let $m,n \ge 2$ be integers not divisible by the characteristic of $K$ and let $(C,\rho)$ be an $n$-covering of its Jacobian $E$. There is a short exact sequence:

\begin{align}
\label{mn1}
0 \to E[m] \to E[mn] \stackrel{m}\to E[n] \to 0
\end{align}
This gives rise to an exact sequence of Galois cohomology groups:
\begin{align}
\label{relatemn}
0 \to \frac{E(K)[n]}{mE(K)[mn]} \rightarrow \HH^1(K,E[m]) \stackrel{i_*}\rightarrow \HH^1(K,E[mn])
\stackrel{m_*}\rightarrow \HH^1(K,E[n])\,.
\end{align} Recall the map $\Psi_\rho:\Cov^{(m)}(C/K) \to \HH^1(K,E[mn])$ given by composing covering maps.

\begin{Lemma}
The image of the composition \[ \Cov^{(m)}(C/K)  \stackrel{\Psi_\rho}\To \HH^1(K,E[mn]) \To \HH^1(K,E)[mn] \] is equal to $\{ \, D \in \HH^1(K,E)[mn] \,:\, mD = C \,\}\,.$ Suppose further that $K$ is a number field and $C$ is everywhere locally solvable. Then the image of the composition \[ \Sel^{(m)}(C/K)  \stackrel{\Psi_\rho}\To \Sel^{(mn)}(E/k) \To \Sha(E/k)[mn] \] is equal to $\{ \, D \in \Sha(E/k)[mn] \,:\, mD = C \,\}\,.$
\end{Lemma}

\begin{Remark}
From the exactness in (\ref{relatemn}) one sees that the fibers of $\Psi_\rho$ are parameterized by the finite group $\frac{E(K)[n]}{mE(K)[mn]}$. In this way one reduces the study of $mn$-coverings of $E$ to the study of $m$-coverings of the $n$-coverings of $E$. When $m$ and $n$ are relatively prime, the sequence (\ref{mn1}) splits and so $\HH^1(K,E[mn]) \simeq \HH^1(K,E[m]) \times \HH^1(K,E[n])$. In this way one can further reduce the problem to the study of $n$-coverings where $n$ is a prime power. In particular, to do an $n^2$-descent on an elliptic curve $E$ for some square free $n$ it suffices to do $p$-descents on $E$ and then compute the $p$-Selmer sets of the elements of the $p$-Selmer group for the primes $p$ dividing $n$.
\end{Remark}

\subsection{Notation for the sequel}
\label{Notation2}
In the remainder of this paper we specialize to the following situation. We consider a $p$-covering $\rho:C\to E$ where $p$ is an odd prime and, unless expressly stated otherwise, make the following assumptions on $C$.
\begin{itemize}
\item $\Pic(C) = \Pic(\bC)^{G_K}$, i.e. every $K$-rational divisor class can be represented by some $K$-rational divisor.
\item $\Cov^{(p)}(C/K) \ne \emptyset$, i.e. there exists a $p$-covering of $C$ defined over $K$.
\end{itemize}
 
These assumptions are satisfied when $K$ is a number field and $C$ is everywhere locally solvable. The first is a result of Cassels \cite[Theorem 1.2]{CaIV}; it is a consequence of the local-global principle for the Brauer group of $K$. The second is a result of Tate (appearing in the same article of Cassels, Lemma 6.1). It is ultimately a consequence of the local-global principle for $\HH^1(K,E[p])$.

From the first assumption above it follows that $(C,\rho)$ has trivial obstruction. The Brauer-Severi diagram corresponding to $(C,\rho)$ gives a model for $C$ as a genus one normal curve of degree $p$ in $\PP^{p-1}$. We fix defining equations of the following form.  For $p=3$, $C\subset \PP^2$ is defined by the vanishing of some ternary cubic form $U(u_1,u_2,u_3) \in K[u_1,u_2,u_3]$. For larger $p$, the model is as a (noncomplete) intersection of $p(p-3)/2$ quadrics $Q_i(u_1,\dots,u_p) \in K[u_1,\dots,u_p]$. We denote the $G_K$-set of flex points on $C$ by $X$.  The corresponding \'etale $K$-algebra, $\Map_K(X,\bK)$ will be denoted by $F$ (for {\em flex algebra}).

\section{Affine structure}
\label{AffineStructure}
\subsection{Affine Maps}
\label{affinemaps}
We maintain the notation laid out in Section \ref{Notation2}. Since $X$ is a $K$-torsor under $E[p]$, we may consider it to be the affine space underlying the $2$-dimensional $\F_p$-vector space $E[p]$. The action of $G_K$ on $X$ factors through the affine general linear group, which is an extension of the general linear group by the group of translations:
\[ 1 \to E[p] \to \AGL(X) \to \GL(E[p]) \to 1\,. \]
Here $E[p]$ acts on $X$ by translations and $\GL(E[p])$ acts on $E[p]$ in the obvious way.

In general, if $V,W$ are vector spaces and $\A$ denotes the affine space underlying $V$, then a map $\phi: \A \to W$ is said to be affine if, for all $x \in \A$ and $P,Q \in V,$ one has \[ \phi(x+P+Q)+\phi(x) = \phi(x+P)+\phi(x+Q)\,.\] Geometrically, this says that the sums of the values of $\phi$ on the two pairs of opposite vertices of any parallelogram in $\A$ are equal. In characteristic different from $2$ it is easy to check that a map is affine if and only if this condition is satisfied for all degenerate parallelograms where one pair of opposite vertices coincide. In other words, 
\begin{align}
\label{affinecond}
\text{$\phi$ is affine} \Longleftrightarrow \forall\,x \in \A\,, P \in V,\,\, \phi(x+P) + \phi(x - P) =  2\phi(x)\,.
\end{align}

We define $\Aff(\A,W)$ to be the vector space of affine maps from $\A$ to $W$. Given an affine map $\phi\in \Aff(\A,W)$ and $x \in \A$, we can obtain a linear map $\Lambda_{\phi,x}:V \to W$ by projecting onto the linear part. This is defined by $\Lambda_{\phi,x}(P)=\phi(x+P) - \phi(x)$. One can easily check that $\Lambda_{\phi,x}$ is linear and does not depend on the choice for $x$. This gives rise to a surjective linear map $\Aff(\A,W) \ni \phi \mapsto \Lambda_{\phi,x} \in \Hom(V,W)$, which fits in an exact sequence
\[ 0 \to W \to \Aff(\A,W) \to \Hom(V,W) \to 0\,.\]

Now return to the case $V = E[p]$ and $\A = X$. We consider $\mu_p$ as an $\F_p$-vector space written multiplicatively. It naturally embeds in $\Aff(X,\mu_p)$ as the subspace of constant maps. The group $\Aff(X,\mu_p)$ itself may be identified with a subgroup of $\Map(X,\bK)$. As such it inherits a natural action of $G_K$. We have a short exact sequence of $G_K$-modules
\begin{align}
\label{Affseq}
1 \to \mu_p \to \Aff(X,\mu_p) \to \Hom(E[p],\mu_p) \to 0 \,.
\end{align} The $G_K$-module $E[p]$ is self-dual via the Weil pairing. Namely, we can identify $E[p]$ with $\Hom(E[p],\mu_p)$ via \[ E[p] \ni P \mapsto e_p(P,-) \in \Hom(E[p],\mu_p)\,.\] 
\vspace{1mm}
\begin{Remark}
Alternatively, one can make this identification using $P \leftrightarrow e_p(-,P)$. Since the Weil pairing is alternating, the two differ by a sign. This controls the factor of $-1$ in the next lemma. We have made our choice in deference to the formulation of Proposition \ref{Phiphs} below.
\end{Remark}

Making this identification in the exact sequence (\ref{Affseq}) above and taking Galois cohomology we obtain an exact sequence 
\begin{align}
\label{DefUps}
\HH^1(K,\mu_p) \to \HH^1(K,\Aff(X,\mu_p)) \to \HH^1(K,E[p]) \stackrel{\Upsilon}{\To} \Br(K)[p]\,.
\end{align} Here we have also identified $\HH^2(K,\mu_p)$ with the $p$-torsion in the Brauer group of $K$. Recall that $C^\perp$ is the subgroup $\{\, \xi \in \HH^1(K,E[p])\,:\, C \cup_p \xi = 0 \,\}$. The next lemma identifies this with the kernel of $\Upsilon$.

\begin{Lemma}
\label{cup}
$\Upsilon(\xi) = -C \cup_p \xi$.
\end{Lemma}

\begin{Remark}
We may consider $\Cov^{(p)}(C/K)$ as the affine space underlying the $\F_p$-vector space $\HH^1(K,E[p])$. With this interpretation the obstruction map $\Ob_{p^2}:\Cov^{(p)}(C/K) \to \Br(K)[p]$ is affine, as one can see from Lemma \ref{Obisaffine}. Lemma \ref{cup} identifies $\Upsilon$ (up to sign) as the corresponding linear map obtained by projecting.
\end{Remark}

\begin{Proof}
Recall that $\rho:C\to E$ denotes the covering map. Let $\psi:C\to E$ be an isomorphism (defined over some extension of $K$) such that $p\circ\psi = \rho$. For any $\sigma \in G_K$, the map $\psi^\sigma - \psi$ corresponds to translation by an element of $E[p]$. This defines a cocycle representing the class of $C$ in $\HH^1(K,E[p])$. The cup product $-C \cup_p \xi$ is the class of the $2$-cocycle 
\[ G_K\times G_K \ni (\sigma,\tau) \mapsto e_p(\psi - \psi^\tau,\,\xi^\tau_\sigma) \in \mu_p\,.\]

Now let $\xi \in \HH^1(K,E[p])$. $\Upsilon$ is a connecting homomorphism, so to compute $\Upsilon(\xi)$ we first choose a lift of $\xi$  to a cochain with values in $\Aff(X,\mu_p)$. For any $P \in E[p]$, we claim that the map $\phi_P:X \ni x \mapsto e_p(P,\psi(x)) \in \mu_p$ is affine and that its image under $\Aff(X,\mu_p)\to E[p]$ is $P$. To see that it is affine, let $x \in X$ and $Q,R \in E[p]$. Using bilinearity of the Weil pairing we have
\begin{align*}	\phi_P(x+Q+R)\cdot\phi_P(x)
		&= e_p(P,\psi(x+Q+R))\cdot e_p(P,\psi(x))\\
		&= e_p(P,\psi(x)+Q+R)\cdot e_p(P,\psi(x))\\
		&= e_p(P,\psi(x)+Q)\cdot e_p(P,\psi(x)+R)\\
		&= \phi_P(x+Q)\cdot\phi_P(x+R)\,.
\end{align*}
The image of $\phi_P$ in $\Hom(E[p],\mu_p)$ is given by projecting onto the linear part. This is the map \[R \mapsto \phi_P(x+R)/\phi_P(x) = e_p(P,\psi(x)+R)/e_p(P,\psi(x)) = e_p(P,R)\,.\] The identification of $E[p]$ with $\Hom(E[p],\mu_p)$ is given by 
\[ E[p] \ni P \leftrightarrow e_p(P,-)\in \Hom(E[p],\mu_p)\,,\] so the image of $\phi_P$ in $E[p]$ is $P$.

Thus $\Upsilon(\xi)$ is given by the coboundary of the cochain $\sigma \mapsto e_p(\xi_\sigma,\psi) = e_p(-\psi,\xi_\sigma) \in \Aff(X,\mu_p)$. Here $e_p(-\psi,\xi_\sigma)$ is the map $x \mapsto e_p(-\psi(x),\xi_\sigma)$.  The value of the coboundary on a pair $(\sigma,\tau) \in G_K\times G_K$ is given by
\[ \frac{ e_p(-\psi,\xi_\sigma)^\tau\cdot e_p(-\psi,\xi_\tau)}{e_p(-\psi,\xi_{\sigma\tau})} = \frac{ e_p(-\psi^\tau,\xi_\sigma^\tau)}{e_p(-\psi,\xi_\sigma^\tau)} = e_p(\psi -\psi^\tau,\xi_\sigma^\tau)\,. \] This is the same as the cup product computed above, so the lemma is proven.
\end{Proof}

\subsection{Making cohomology groups explicit}
We have identified $C^\perp$ with the kernel of $\Upsilon$. By exactness of the sequence (\ref{DefUps}) defining $\Upsilon$, this is the same as the image of $\HH^1(K,\Aff(X,\mu_p))$ in $\HH^1(K,E[p])$. This suggests that we should look for a practical description of $\HH^1(K,\Aff(X,\mu_p))$.

Recall that $F$ denotes the flex algebra, $\Map_K(X,\bK)$, and that $\mu_p(\bF) = \Map(X,\mu_p)$. We have a canonical monomorphism $\Aff(X,\mu_p) \hookrightarrow \mu_p(\bF)$; this is simply the observation that an affine map is a map. To obtain a description of $\HH^1(K,\Aff(X,\mu_p))$ we want to extend this to a short exact sequence and take its Galois cohomology. To do this, we want to write down a morphism on $\mu_p(\bF)$ whose kernel is $\Aff(X,\mu_p)$. We can achieve this by identifying a suitable $G_K$-set derived from $X$ and taking the {\em induced norm map} described in Section \ref{DerivedGK}. A $G_K$-set derived from $X$ is just a $G_K$-stable set of divisors on $C$ that are supported entirely on $X$. This and the characterization of affine maps in (\ref{affinecond}) motivates the following lemma.
\begin{Lemma}
\label{hyperplanes}
The set of divisors of the form $(p-2)[x] + [x+P] + [x-P] \in \Div(\bC)$, with $x \in X$ and $P \in E[p]$, is a $G_K$-stable set of hyperplane sections of $C$.
\end{Lemma}

\begin{Proof}
The fact that these divisors form a $G_K$-set follows from the fact that the action of $E[p]$ on $X$ is Galois equivariant. To see that they are hyperplane sections, we may work geometrically, considering $C$ as an elliptic curve with some flex $x_0 \in X$ as distinguished point. Note that the flex points are then the $p$-torsion points on the elliptic curve $(C,x_0)$. Since the model for $C$ is given by the embedding corresponding to the complete linear system $|p[x_0]|$, the hyperplane sections are precisely those divisors linearly equivalent to $p[x_0]$. That the divisors in the lemma are hyperplane sections is then a consequence of the well-known fact that two divisors on an elliptic curve are linearly equivalent if and only if they have the same degree and the same sum.
\end{Proof}

We use $Y$ to denote the set of hyperplanes in Lemma \ref{hyperplanes}. It is a $G_K$-set and we denote its corresponding \'etale $K$-algebra by $H$ (for {\em hyperplane algebra}). Note that $Y$ is derived from $X$ in the sense described in Section \ref{GKEtale}. Using (\ref{affinecond}) we see that the induced norm map $\partial: F \to H$ yields an exact sequence
\begin{align}
1 \to \Aff(X,\mu_p) \To \bF^\times \stackrel{\partial}\To \bH^\times\,.
\end{align}
This allows us to obtain the desired description of $\HH^1(K,\Aff(X,\mu_p))$.

\begin{Lemma}
\label{dfp}
\[\HH^1(K,\Aff(X,\mu_p)) \simeq \frac{(\partial\bF^\times)^{G_K}}{\partial F^\times}\,.\] 
\end{Lemma}

\begin{Proof}
We have an exact sequence \[ 1\to\Aff(X,\mu_p) \to \bF^\times \stackrel{\partial}\to \partial\bF^\times \to 1\,.\] By Hilbert's Theorem 90, $\HH^1(K,\bF^\times) = 0$. The result follows by considering the long exact sequence of Galois cohomology groups.
\end{Proof}

\begin{Corollary}
\label{desc2}
There is an isomorphism $C^\perp \simeq (\partial\bF^\times)^{G_K}/K^\times\partial F^\times$, where we identify $K^\times$ with its image in $H^\times$.
\end{Corollary}

\begin{Proof}
Lemma \ref{cup} shows that $C^\perp$ is isomorphic to $\HH^1(K,\Aff(X,\mu_p))$ modulo the image of $\HH^1(K,\mu_p)$. Lemma \ref{dfp} and Hilbert's Theorem 90 allow us to identify these groups with $(\partial\bF^\times)^{G_K}/\partial F^\times$ and $K^\times/K^{\times p}$, respectively. We only need to show that the identifications are compatible.

Noting that $\bK \subset \bF = \Map(X,\bK)$ consists of the constant maps we see that for $\alpha \in \bK$, $\partial(\alpha)= \alpha^p \in \bK \subset \bH$. This shows that the following diagram commutes:
\[ \xymatrix{	
	1 \ar[r]& \mu_p \ar[r]\ar@{^{(}->}[d]& \bK^\times \ar[r]^p\ar@{^{(}->}[d]& \bK^\times \ar[r]\ar@{^{(}->}[d]& 1\\
	1 \ar[r]& \Aff(X,\mu_p) \ar[r]& \bF^\times \ar[r]^\partial& \partial\bF^\times \ar[r]& 1 }\]
The identifications are made by taking Galois cohomology and using that, by Hilbert's Theorem 90, $\HH^1(K,-)$ of the middle terms vanish. From this compatibility is clear.
\end{Proof}

Before proceeding we fix some notation. As a $G_K$-set, $Y$ splits as a disjoint union of (at least) two $G_K$-stable subsets. The first consists of the $p^2$ hyperplane sections of the form $p[x]$ with $x \in X$. These divisors correspond to pairs $(x,P)$ with $P=0$. The other consists of those $y \in Y$ associated to some pair $(x,P)$ with $P \ne 0$. These two $G_K$-subsets will be denoted by $Y_1$ and $Y_2$; their corresponding \'etale $K$-algebras will be denoted by $H_1$ and $H_2$. As $G_K$-sets, $X$ and $Y_1$ are isomorphic and so we will identify $F$ with $H_1$. Thus $H$ splits as $H \simeq H_1\times H_2\simeq F\times H_2$. Correspondingly we have a splitting of the induced norm map as $\partial = \partial_1\times\partial_2$ with $\partial_1: F \ni \alpha \mapsto \alpha^p \in F$.

For $p=3$, $Y_2$ consists of the $12$ lines of $\PP^2$ that pass through three distinct flex points of $C$. For $p\ge 5$, $X\times \frac{E[p]\setminus\{0_E\}}{\{\pm1\}}$ and $Y_2$ are isomorphic as $G_K$-sets, a pair $(x,P)$ corresponding to the hyperplane section $(p-2)[x]+[x+P]+[x-P]$. From this we see that $\#Y_2 = p^2(p^2-1)/2$. There is a canonical projection $Y_2 \ni (x,P) \mapsto x \in X$. Thus, for $p\ge 5$, $H_2$ may be viewed as an $F$-algebra of degree $(p^2-1)/2$.

\section{The descent map}
\label{TheDescentMap}
Recall (Corollary \ref{Covphs}) that $\Cov_0^{(p)}(C/K)$ is a principal homogeneous space for $C^\perp$. We have obtained a more or less concrete description of $C^\perp$ as a subgroup of $H^\times/K^\times\partial F^\times$. The goal now is to give an equally explicit description of $\Cov_0^{(p)}(C/K)$ as a coset of $C^\perp$ inside $H^\times/K^\times\partial F^\times$. This will be achieved by our {\em descent map}. 

Let $\x \in \Div(C\otimes_KF) = \Map_K(X,\Div(\bC))$ denote the map whose value at $x \in X$ is the divisor $[x]$. Similarly we use $\y$ to denote the element of $\Div(C\otimes_KH)= \Map_K(Y,\Div(\bC))$ whose value at $y \in Y$ is the divisor $y \in \Div(\bC)$. By Lemma \ref{hyperplanes}, the divisors in $Y$ are all hyperplane sections of $C$. Thus we can choose a linear form $\tilde{\ell} \in H[u_1,\dots,u_p]$ cutting out the divisor $\y$. We then choose any linear form $u \in K[u_1,\dots,u_p]$ cutting out a divisor on $C$ that is disjoint from $X$ to obtain a rational function $\ell = \tilde{\ell}/u \in \kappa(C\otimes_K H)^\times$ with \[ \diw(\ell) = \y - \diw(u)\,.\]

\begin{Proposition}
The function $\ell$ induces a unique homomorphism 
\[ \Phi:\Pic(C) \to H^\times/K^\times\partial F^\times\,,\] with the property that the image of any divisor class is given by evaluating $\ell$ at any $K$-rational representative with support disjoint from $X$ and $\diw(u)$.
\end{Proposition}

\begin{Proof}
This follows directly from Proposition \ref{DescentOnPic}.
\end{Proof}

Recall that we have assumed $\Pic(C) = \Pic(\bC)^{G_K}$. Identifying $\Pic^1(C)$ with $C(K)$ we can think of $\Phi$ as giving a map on the $K$-points of $C$. For the points outside $X$ and $\diw(u)$, this is simply given by evaluating $\ell$. Note also that the homomorphism does not depend on the choice for $u$. So if we like, we may determine the image of a point by evaluating $\tilde{\ell}$ on some choice of homogeneous coordinates. For this reason we may refer to $\tilde{\ell}$ as the linear form defining the descent map in the following theorem.

\begin{Theorem}
\label{PhionCov}
The choice of linear form $\tilde{\ell}$ determines a unique well-defined map (called the {\em descent map})
\[ \tilde{\Phi} : \Cov_0^{(p)}(C/K) \To H^\times/K^\times \partial F^\times\] with the following property. If $(D,\pi)\in \Cov^{(p)}_0(C/K)$ and $K\subset L$ is any extension of fields with $Q \in D(L)$, then
\[ \tilde{\Phi}((D,\pi)) \equiv \Phi(\pi(Q)) \Mod L^\times \partial F_L^\times\,. \] In particular, if $D(K) \ne \emptyset$, then $\tilde{\Phi}((D,\pi))$ is the image of some $K$-rational point of $C$ under $\Phi$.
\end{Theorem}

\begin{Remark}
Recall that $\Cov_0^{(p)}(C/K)$ yields a partition of the $K$-rational points of $C$,
\[ C(K) = \coprod_{(D,\pi)\in \Cov_0^{(p)}(C/K)} \pi(D(K))\,. \] The defining property says that $\Phi : C(K) \to H^\times/K^\times \partial F^\times$ is constant on each of the sets appearing in this partition and that the value on each is equal to the image of the corresponding covering under the descent map.
\end{Remark}

\begin{Proof}
Let $(D,\pi) \in \Cov_0^{(p)}(C/K)$. By Lemma \ref{pbflex} we have a model for $(D,\pi)$ as a genus one normal curve of degree $p^2$ in $\PP^{p^2-1} = \PP^{p^2-1}(z_1:\dots:z_{p^2})$, where $\pi$ is defined by homogeneous polynomials, $\pi_i \in K[z_1,\dots,z_{p^2}]$ and the pull-back of any flex point $x$ on $C$ is a hyperplane section $\mathfrak{h}_x$ of $D$. For any $x$, $\mathfrak{h}_x$ can be defined by the vanishing of some linear form $h_x\in \bK[z_1,\dots,z_{p^2}]$. Moreover, we can choose these $h_x$ to form a Galois equivariant family. Thus they may be patched together to obtain a linear form $h \in F[z_1,\dots,z_{p^2}]$ cutting out the divisor $\pi^*\x$ on $D\otimes_K F$.

Since the zero divisor of $\ell$ is $\y = \partial\x \in \Div(C\otimes_K H)$ we see that $\partial h$ and $\tilde{\ell}\circ\pi$ cut out the same divisor on $D$. It follows that the rational functions $\ell\circ\pi$ and $\partial h/u\circ\pi$ have the same divisor. Hence there exists some $\Delta \in H^\times$ such that
\begin{align}
\label{rel}
\ell \circ \pi = \Delta\cdot\left(\frac{\partial h}{u \circ \pi}\right) \text{ in $\kappa(D\otimes_KH)^\times$.}
\end{align} We define the image of $\tilde{\Phi}((D,\pi))$ to be the class of $\Delta \in H^\times/K^\times\partial F^\times$.

A different choice of forms defining $\pi$ would change the left-hand side of (\ref{rel}) by a factor in $K^\times$.  Similarly, a different choice for the form $h= (h_x)_{x \in X}$ defining $(\mathfrak{h}_x)_{x \in X}$ would change the right-hand side of (\ref{rel}) by a factor in $\partial F^\times$. Thus, having fixed the model for $(D,\pi)$ in $\PP^{p^2-1}$ we get a well-defined element of $H^\times/K^\times\partial F^\times$. 

Let us show that this does not depend on the model. Suppose $(D',\pi')$ is isomorphic to $(D,\pi)$, and choose a model for $(D',\pi')$ in $\PP^{p^2-1}$ as genus one normal curve. As above, choose a linear form $h' \in F[z_1,\dots,z_{p^2}]$ cutting out the divisors $\pi'^*\x$ on $D'$. By assumption we have an isomorphism of coverings $\varphi : D' \rightarrow D$ defined over $K$ (i.e. such that $\pi'=\pi\circ\varphi$). Let $\varphi^* : \kappa(D\otimes_KH) \rightarrow \kappa(D'\otimes_KH)$ denote the isomorphism of function fields induced by $\varphi$. Applying $\varphi^*$ to equation (\ref{rel}), we obtain a relation in $\kappa(D'\otimes_KH)$, 
\begin{align}
\label{stringeq}
\Delta\cdot \left(\frac{\partial(h\circ\varphi)}{ u\circ\pi'}\right) = \varphi^*\bigl(\Delta\cdot\left(\frac{\partial h}{ u\circ\pi}\right) \bigr)= \varphi^*(\ell\circ\pi) = \ell\circ\pi\circ\varphi = \ell\circ\pi'\,.
\end{align}
The divisor on $D'$ cut out by $h\circ \varphi$ is $\pi'^*\x$, so the extremal terms in (\ref{stringeq}) define the image of $(D',\pi')$ under the descent map. Thus the image of $(D',\pi')$ is also the class of $\Delta$, which shows that $\tilde{\Phi}$ is well-defined.

It remains to show that $\tilde{\Phi}$ has the stated property. For this let $Q \in D(L)$ be a point defined over some extension $L$ of $K$. We can find an $L$-rational divisor $d = \sum_in_iQ_i$ on $D$ linearly equivalent to $[Q]$ and such that the support of $d$ contains no points lying above the flex points of $C$ or the zeros of $u$. The divisor $[\pi(Q)]$ on $C$ is linearly equivalent to the $L$-rational divisor $\pi_*d := \sum_in_i[\pi(Q_i)]$ (e.g. \cite[II.3.6]{Silverman}). So $\Phi(\pi(Q))$ is represented by $\ell(\pi_*d)$. On the other hand, the relation (\ref{rel}) defining $\Delta$ gives, \[ \ell(\pi_*d) =  \Delta\cdot \left(\frac{\partial h}{u\circ\pi}\right)(d)\,,\] since $\deg(d) = 1$. Now since $d$ is $L$-rational, $\left(\frac{\partial h}{u\circ\pi}\right)(d) \in L^\times \partial F_L^\times$. So $\Phi(\pi(Q))$ is represented by $\Delta$ as required.
\end{Proof}

In what follows we will refer to a linear form $h \in F[z_1,\dots,z_{p^2}]$ as in the proof (i.e. such that $\pi^*\x = \diw(h)$) as a linear form defining the pull-back of a generic flex. Recall that $\ell$ is defined as the ratio $\ell = \tilde{\ell}/u$. If $(D,\pi) \in \Cov_0^{(p)}(C/K)$ and $h$ is a linear form defining the pull-back of a generic flex (on some model of $(D,\pi)$), then the image of $(D,\pi)$ under the descent map is also represented by the $\Delta \in H^\times$ satisfying the relation \[ \tilde{\ell}\circ\pi = \Delta \partial h \] in the coordinate ring of $D$.

\subsection{Injectivity of the descent map}
\label{InjectivityOfTheDescentMap}
The following proposition shows that the descent map respects the affine structure.
\begin{Proposition}
\label{Phiphs}
Let $(D,\pi) \in \Cov_0^{(p)}(C/K)$, $\xi \in C^\perp$ and $(D,\pi)\cdot \xi$ be the twist of $(D,\pi)$ by $\xi$. Then \[ \tilde{\Phi}((D,\pi)\cdot\xi) = \tilde{\Phi}((D,\pi))\cdot\tilde{\Phi}_0(\xi) \in H^\times/K^\times\partial F^\times\,, \] where $\tilde{\Phi}_0:C^\perp \simeq (\partial\bF^\times)^{G_K}/K^\times\partial F^\times$ is the isomorphism given by Corollary \ref{desc2}.
\end{Proposition}

Recall that $C^\perp$ acts simply transitively on $\Cov_0^{(p)}(C/K)$ by twisting. The proposition shows that the image of $C^\perp$ under $\tilde{\Phi}_0$ acts on the image of $\Cov_0^{(p)}(C/K)$ under $\tilde{\Phi}$ by multiplication and that the pair $(\tilde{\Phi},\tilde{\Phi}_0)$ respects these two actions. Since $\tilde{\Phi}_0$ is an isomorphism, we deduce the following.

\begin{Corollary}
\label{Injective}
Assume $\Cov_0^{(p)}(C/K)$ is nonempty. Then the descent map is an affine isomorphism (i.e. isomorphism of principal homogeneous spaces). In particular $\tilde{\Phi}:\Cov_0^{(p)}(C/K) \to H^\times/K^\times\partial F^\times$ is injective, and its image is a coset of $(\partial\bF^\times)^{G_K}/K^\times\partial F^\times$ inside $H^\times/K^\times\partial F^\times$.
\end{Corollary}

Before giving the proof, it will be useful to put together a diagram. For any $x_0 \in X$ and $P \in E[p]$, the Weil pairing on $E[p]$ gives a map
\[ \phi_{P,x_0}: X \ni x \mapsto e_p(P,x-x_0) \in \mu_p\,, \] where $x-x_0$ denotes the unique $T \in E[p]$ such that $x_0 + T = x$. A different choice for $x_0$ gives a map which differs by a constant factor. Thus, the image of $\phi_{P,x_0}$ in  \[\mu_p(\bF)/\mu_p=\Map(X,\mu_p)/\{\text{constant maps}\}\] depends only on $P$. Nondegeneracy of the Weil pairing shows that distinct choices for $P$ lead to distinct maps. Thus we have an embedding $E[p] \hookrightarrow \mu_p(\bF)/\mu_p$.

Recall that the kernel of $\partial|_{\bF^\times}$ is the space of affine maps to $\mu_p$. Since $\partial_1$ is just the $p$-th power map, the space of affine maps is also equal to the kernel of $\partial_2|_{\mu_p(\bF)}$. Since the constant maps are affine, $\partial_2$ induces a map on $\mu_p(\bF)/\mu_p$. For any $P \in E[p]$ and $x_0 \in X$, the map $\phi_{P,x_0}$ is affine (cf. the proof of \ref{cup}). Now, by counting dimensions for example, we see that the sequence \[0 \to E[p] \to \mu_p(\bF)/\mu_p \stackrel{\partial_2}{\To} \mu_p(\bH_2)\] is exact.

We also have an exact sequence \[1 \to \Aff(X,\mu_p) \to \mu_p(\bF) \stackrel{\partial_2}{\To} \mu_p(\bH_2)\,.\] We claim that the two are compatible in the sense that following diagram commutes. The map $\Aff(X,\mu_p) \to E[p]$ is given by projecting an affine map onto its linear part and then identifying $E[p]$ with its dual using the Weil pairing (so the vertical sequence on the left is (\ref{Affseq}), considered in the discussion leading up to Lemma \ref{cup}).
\begin{align}
\label{diagram1}
\xymatrix{ 	& 1 \ar[d] & 1\ar[d] & \\
		& \mu_p \ar[d]\ar@{=}[r] & \mu_p\ar[d]& \\
		1\ar[r]&\Aff(X,\mu_p)\ar[r]\ar[d]&\mu_p(\bF)\ar[d]\ar[r]^{\partial_2}& \partial_2\bigl(\mu_p(\bF)\bigr)\ar@{=}[d]\ar[r]&1\\
		0\ar[r]&E[p]\ar[r]\ar[d]& \mu_p(\bF)/\mu_p\ar[d]\ar[r]^{\partial_2}&\partial_2\bigl(\mu_p(\bF)\bigr)\ar[r]&1\\
		& 0 & 1\\}
\end{align}
Note that the rows and columns are exact.

\begin{Lemma}
Diagram (\ref{diagram1}) commutes.
\end{Lemma}

\begin{Proof}
We only need to show that the lower-left square commutes, the rest being obvious. Let $\phi:x\mapsto\phi(x)$ be an affine map. Choose some $x_0 \in X$. Projection onto the linear part is the map $\Lambda_\phi:P \mapsto \phi(x_0+P)/\phi(x_0)$. Identifying $E[p]$ with its dual via the Weil pairing, $\Lambda_\phi$ is the unique $R \in E[p]$ such that, for all $P \in E[p]$, $\phi(x_0+P)/\phi(x_0) = e_p(R,P)$. The image of this $R$ in $\mu_p(\bF)/\mu_p$ is the class of the map $\phi_{R,x_0}:x\mapsto e_p(R,x-x_0)$. By the property defining $R$, this is equal to the map $x \mapsto \phi(x_0 + (x-x_0))/\phi(x_0) = \phi(x)/\phi(x_0)$. Modulo constant maps, this is the same as the image of $\phi$ in $\mu_p(\bF)$, so the diagram commutes.
\end{Proof}

For the proof of Proposition \ref{Phiphs}, we will make use of an alternative description of the embedding $E[p] \hookrightarrow \mu_p(\bF)/\mu_p$.

\begin{Lemma}
\label{WPLemma} Let $D \in \Cov_0^{(p)}(C/\bK)$ ({\sc nb}: over $\bK$, not $K$) and let $h$ denote a linear form (with coefficients in $\bF$) defining the pull-back of the generic flex point on $C$. For any $R \in E[p]$, the image of $R$ under the composition $E[p] \hookrightarrow \mu_p(\bF)/\mu_p \hookrightarrow \bF^\times/\bK^\times$ is equal to the class of $\frac{h(Q+R)}{h(Q)}$, where $Q \in D$ is any point chosen so that $h(Q)$ and $h(Q+R)$ are both defined and nonzero.
\end{Lemma}

\begin{Proof}
Let $\psi:C \to E$ be the isomorphism (defined over $\bK$) such that $p\circ\psi = \rho$ and let $x_0$ be a preimage of $0_E$ under $\psi$. Further, let $Q_0$ be any preimage of $x_0$ on $D$ and $\psi_D:D \to E$ be the isomorphism defined (over $\bK$) by $Q \mapsto (Q-Q_0)$. We have a commutative diagram,
\[ \xymatrix{
	E \ar[d]^{p} & D \ar[d]^{\pi} \ar[l]_{\psi_D} \\
	E &		  C \ar[l]^{\psi} \\
}\]
If $x$ is a flex point, evaluating the coefficients of $h$ at $x$ gives a linear form $h_x$ defining the pull-back of $[x]$ by $\pi$. Consider the function $h_x/h_{x_0} \in \kappa(\bar{D})^\times$ and its image $g_x = (\psi_D^{-1})^*(h_x/h_{x_0}) \in \kappa(\bar{E})^\times$. The divisor of $h_x/h_{x_0}$ is $\pi^*[x]-\pi^*[x_0]$, so by commutativity $\diw(g_x) = p^*[(x-x_0)]-p^*[0_E]$. By definition of the Weil pairing \cite[III.8]{Silverman}, for any $R \in E[p]$,  \[ e_p(R,x-x_0) = \frac{g_x(T + R)}{g_x(T)}\,,\] where $T \in E$ is any point chosen so that both numerator and denominator are defined and nonzero. Thus, we have 
\begin{align}
\label{eqqq}
e_p(R,x-x_0) = \frac{h_x(\psi_D^{-1}(T)+R)h_{x_0}(\psi_D^{-1}(T))}{h_x(\psi_D^{-1}(T))h_{x_0}(\psi_D^{-1}(T)+R)}\,. \end{align} Considered as an element of $\bF^\times = \Map(X,\bK^\times)$ modulo the constant maps, the right-hand side of (\ref{eqqq}) is represented by the map \[ \frac{h(\psi_D^{-1}(T)+R)}{h(\psi_D^{-1}(T))} = \left(x \mapsto \frac{h_x(\psi_D^{-1}(T)+R)}{h_x(\psi_D^{-1}(T))}\right)\,.\] On the other hand, the left-hand side of (\ref{eqqq}) represents the image of $R$ in $\mu_p(\bF)/\mu_p$, so we are done. \end{Proof}

{\bf Proof of Proposition \ref{Phiphs}}
Let $(D,\pi)$, $(D_\xi,\pi_\xi)$ and $\xi$ be as in the proposition, and fix models for everything in $\PP^{p^2-1}$. We have an isomorphism (of coverings) $\varphi:D_\xi \to D$ defined over $\bK$, with the property that $\varphi^\sigma(Q) = \varphi(Q)+\xi_\sigma$ for all $Q \in D_\xi$ and $\sigma \in G_K$.

Choose linear forms $h$ and $h_\xi$ with coefficients in $F$ defining the pull-backs of the generic flex by $\pi$ and $\pi_\xi$, respectively. For some $\Delta,\Delta_\xi \in H^\times$, necessarily representing the images of $(D,\pi)$ and $(D_\xi,\pi_\xi)$ under $\tilde{\Phi}$, we have

\[ \Delta\cdot\partial h = \tilde{\ell}\circ \pi \text{ and } \Delta_\xi\cdot\partial h_\xi = \tilde{\ell}\circ \pi_\xi\,\] in the coordinate rings of $D$ and $D_\xi$, respectively. Applying $\varphi^*$ to the first relation and comparing with the second gives \[ \Delta\cdot\partial(h\circ\varphi) = \Delta_\xi\cdot\partial h_\xi \] in the coordinate ring of $D_\xi$. Specializing to a point $Q$ in $D_\xi$ not lying above any flex point of $C$ (i.e. so that neither $h_\xi$ nor $h\circ\varphi$ vanish at $Q$) we have \[ \frac{\Delta_\xi}{\Delta} = \partial\Bigl(\frac{h(\varphi(Q))}{h_\xi(Q)}\Bigr) \in (\partial{\bF^\times})^{G_K}\,.\] Note that $h_\xi(Q)$ and $h(\varphi(Q))$ depend on a choice of homogeneous coordinates for $Q$, but that their ratio does not. This is $G_K$-invariant since $\Delta$ and $\Delta_\xi$ are in $H^\times$.

Under the isomorphism $(\partial\bF^\times)^{G_K}/\partial F^\times \simeq \HH^1(k,\Aff(X,\mu_p))$ given in Lemma \ref{dfp}, $\partial\Bigl(\frac{h(\varphi(Q))}{h_\xi(Q)}\Bigr)$ corresponds to the class of the cocycle \[ \eta:G_K \ni\sigma \mapsto \left(\frac{h(\varphi(Q))}{h_\xi(Q)}\right)^\sigma\left(\frac{h_\xi(Q)}{h(\varphi(Q))}\right) \in \mu_p(\bF) = \Map(X,\mu_p)\,,\] which a priori takes values in $\Aff(X,\mu_p) \subset \mu_p(\bF)$. We need to show that the image of this cocycle under the map induced by $\Aff(X,\mu_p) \to E[p]$ is cohomologous to $\xi$. For this we make use of the following commutative diagram
\begin{align}
\label{digram}
\xymatrix{ \Aff(X,\mu_p) \ar@{^{(}->}[r]\ar[d]& \mu_p(\bF) \ar@{^{(}->}[r]\ar[d]& \bF^\times \ar[d] \\
		E[p] \ar@{^{(}->}[r]& \mu_p(\bF)/\mu_p \ar@{^{(}->}[r]& \bF^\times/\bK^\times }
\end{align}
Since the horizontal maps are all injective, it will be enough to show that, for any $\sigma \in G_K$, the images of $\xi_\sigma$ and $\eta_\sigma$ in the lower-right corner are equal. For this we will make use of the preceding lemma.

Using the fact that $h$ and $h_\xi$ are defined over $H$ and rearranging, we have \[ \eta_\sigma = \left(\frac{h(\varphi^\sigma(Q^\sigma))}{h(\varphi(Q))}\right) \left(\frac{h_\xi(Q)}{h_\xi(Q^\sigma)}\right)\,.\] Making use of the fact that $\varphi^\sigma(Q^\sigma) = \varphi(Q^\sigma) + \xi_\sigma$ we can rewrite this as 

\[ \eta_\sigma = \left(\frac{h(\varphi(Q) + \xi_\sigma +(\varphi(Q^\sigma)-\varphi(Q))\,)}{h(\varphi(Q))}\right) \left(\frac{h_\xi(Q^\sigma + (Q - Q^\sigma))}{h_\xi(Q^\sigma)}\right)\,.\] By Lemma \ref{WPLemma} this represents the image of \[ \xi_\sigma + (\varphi(Q^\sigma) - \varphi(Q)) - (Q^\sigma-Q) \in E[p]\] under the embedding given by the bottom row of (\ref{digram}). But \[(\varphi(Q^\sigma)-\varphi(Q)) - (Q^\sigma-Q) = 0_E\] (see \cite[X.3.5]{Silverman}) so the images of $\eta_\sigma$ and $\xi_\sigma$ in the lower right corner of (\ref{digram}) are equal. From this the proposition follows.
\hspace*{\fill}\nobreak$\Box$\par\hspace{2mm}

We can use a similar argument to prove the following useful lemma. This says that we could use $\ell$ to perform descent on $E = \Jac(C)$. In practice, one is likely to have produced $C$ by performing a descent on $E$, so this is not going to yield anything new. It does however allow us to relate the descent on $C$ to the descent on $E$.

\begin{Lemma}
\label{DescentOnE} The following diagram is commutative.
\[ \xymatrix{ \Pic^0(C)\ar@{=}[rrr] \ar[d]^{\Phi}  &&& E(K) \ar[d]^{\delta_E}\\ \frac{(\partial\bF^\times)^{G_K}}{K^\times\partial F^\times} \ar[rr]^{\tilde{\Phi}_0^{-1}} 
&& C^\perp \ar@{^{(}->}[r] & \HH^1(K,E[p])}
\]
Here $\delta_E$ is the connecting homomorphism from the Kummer sequence associated to $E$, and the composition of the bottom row is the map identifying $(\partial\bF^\times)^{G_K}/K^\times\partial F^\times$ with $C^\perp \subset \HH^1(K,E[p])$.
\end{Lemma}

\begin{Proof}
Let $\Xi \in \Pic^0(C)$ and choose a representative $d \in \Div(C)$ whose support is disjoint from $X$ and any zeros of $u$. Write $d$ as a difference $d=d_1 - d_2$ of effective divisors and write each $d_i$ as a sum $d_i = \sum_{j=1}^n P_{i,j}$ of $n = \deg(d_1)=\deg(d_2)$ (possibly non-distinct) points on $C$. Now choose any $(D,\pi) \in \Cov_0^{(p)}(C/\bK)$ and a linear form $h$ with coefficients in $\bF$ defining the pull-back of the generic flex. For each $P_{i,j}$ in the support of $d$, choose a point $Q_{i,j} \in D$ such that $\pi(Q_{i,j})=P_{i,j}$. These choices are such that, as points on $E$,
\[ p\bigl( Q_{i,j}-Q_{i',j'} \bigr) = \bigl( P_{i,j}-P_{i',j'}\bigr)\,,\] for any $i,j,i',j'$. In particular, $p\sum_{j=1}^n (Q_{1,j} - Q_{2,j}) = d$. So the image of $\Xi$ under the connecting homomorphism is given by the cocycle
\[ \sigma \mapsto \left(\sum_{j=1}^n (Q_{1,j}^\sigma - Q_{2,j}^\sigma) - \sum_{j=1}^n (Q_{1,j} - Q_{2,j})\right) \in E[p]\,.\]

On the other hand, the image of $\Xi$ under $\Phi$ is represented by 
\[\frac{\ell(d_1)}{\ell(d_2)} =\prod_{j=1}^n\frac{\ell(P_{1,j})}{\ell(P_{2,j})}\,.\]  Choose homogeneous coordinates for the $P_{i,j}$ which are compatible with the action of the Galois group (i.e. so that applying $\sigma$ to the coordinates of $P_{i,j}$ gives the homogeneous coordinates for $P_{i,j}^\sigma$). The class of $\Phi(\Xi)$ is then also represented by \[ \prod_j\frac{\tilde{\ell}(P_{1,j})}{\tilde{\ell}(P_{2,j})} \in H^\times\,,\] where $\tilde{\ell}(P_{i,j})$ now means evaluating the linear form on the given choice of homogeneous coordinates for $P_{i,j}$.

In the coordinate ring of $D$ (over $\bH$) we have $\tilde{\ell}\circ\pi = \tilde{\Phi}((D,\pi))\cdot \partial h$. We can fix forms defining the covering map $\pi$ and choose homogeneous coordinates for the $Q_{i,j}$ so that the equality $\pi(Q_{i,j})=P_{i,j}$ is also true for the coordinates chosen. Now since $\deg(d)=0$, we have that \[ \prod_{j=1}^n\frac{\tilde{\ell}(P_{1,j})}{\tilde{\ell}(P_{2,j})} = \prod_{j=1}^n \frac{\partial h(Q_{1,j})}{\partial h(Q_{2,j})} =  \partial\left(\prod_{j=1}^n \frac{h(Q_{1,j})}{h(Q_{2,j})}\right)  \in \partial \bF^\times\,,\] whereby $h(Q_{i,j})$ means evaluating $h$ at the given choice of coordinates.

Under the isomorphism $(\partial\bF^\times)^{G_K}/K^\times\partial F^\times \simeq \HH^1(K,\Aff(X,\mu_p))/K^\times$, $\Phi(\Xi)$ is sent to the class of the cocycle
\[ \sigma \mapsto \alpha^\sigma/\alpha\,,\] where $\alpha \in \bF^\times$ is any element such that $\partial \alpha$ represents $\Phi(\Xi)$. The argument above shows we may take $\alpha = \left(\prod_{j=1}^n \frac{h(Q_{1,j})}{h(Q_{2,j})}\right)$. Hence, the image of $\Xi$ in $\HH^1(K,\Aff(X,\mu_p))/K^\times$ is represented by the cocycle $\eta$ sending $\sigma \in G_K$ to 
\begin{align*}
\eta_\sigma &= \left(\prod_{j=1}^n \frac{h(Q_{1,j})}{h(Q_{2,j})}\right)^\sigma\cdot \left(\prod_{j=1}^n \frac{h(Q_{2,j})}{h(Q_{1,j})}\right)\\ &= \left(\prod_{j=1}^n \frac{h^\sigma(Q_{1,j}^\sigma)}{h^\sigma(Q_{2,j}^\sigma)}\right)\cdot \left(\prod_{j=1}^n \frac{h(Q_{2,j})}{h(Q_{1,j})}\right)\\
&= \left(\prod_{j=1}^n \frac{h^\sigma(Q_{2,j}^\sigma +(Q_{1,j}^\sigma-Q_{2,j}^\sigma))}{h^\sigma(Q_{2,j}^\sigma)}\right)\cdot \left(\prod_{j=1}^n \frac{h(Q_{1,j}+(Q_{2,j}-Q_{1,j}))}{h(Q_{1,j})}\right)
\end{align*}
Applying Lemma \ref{WPLemma} as in the proof of the proposition, to each factor appearing, we see that the image of $\eta_\sigma$ in $\HH^1(K,E[p])$ is equal to the class of the cocycle \[ G_K\ni \sigma \mapsto \sum_{j=1}^n\left((Q_{1,j}^\sigma-Q_{2,j}^\sigma) - ((Q_{1,j}-Q_{2,j})\right) \in E[p]\,.\] This is the same as the image under the connecting homomorphism, so the diagram commutes.
\end{Proof}

\subsection{Image of the descent map}
\label{ImageOfTheDescentMap}
From here on use $\mathcal{H}_K$ to denote the image of $\Cov_0^{(p)}(C/K)$ under the descent map and $\mathcal{H}_K^0$ for $(\partial\bF^\times)^{G_K}/K^\times\partial F^\times$. We know that $\mathcal{H}_K$ is a coset of $\mathcal{H}_K^0$ inside $H^\times/K^\times\partial F^\times$. Recall also that Corollary \ref{desc2} gives an isomorphism $C^\perp \simeq \mathcal{H}_K^0$. 

In practice one prefers to work with representatives in $H^\times$. We set $\tilde{\mathcal{H}}_K^0 = (\partial{\bF}^\times)^{G_K}\subset H^\times$. To achieve the same for $\mathcal{H}_K$, let $P \in C(\bK)$ be any point that is neither a flex nor a zero of $u$. Define $\tilde{\mathcal{H}}_K = (\ell(P)\cdot \partial{\bF}^\times)^{G_K}$. That this does not depend on the choice for $P$ is shown in the proof below.

\begin{Lemma}
\label{tildeH}
$\mathcal{H}_K=\tilde{\mathcal{H}}_K/K^\times\partial F^\times$
\end{Lemma}

\begin{Proof} To show that $\tilde{\mathcal{H}}_K$ does not depend on $P$, let $P'\in C(\bK)$ be any point which is neither a zero nor a pole of $\ell$. Choose $(D,\pi) \in \Cov_0^{(p)}(\bC/\bK)$ ({\small\sc nb:} over $\bK$, not $K$). Fixing a model for $(D,\pi)$ and choosing a linear form $h$ defining the pullback of the generic flex, we get a relation $\ell\circ\pi = \Delta (\partial h/u\circ\pi)$ in $\kappa(D\otimes_\bK\bH)$. Choosing points $Q$, $Q'$ lying above $P$ and $P'$ we get that \[ \ell(P)/\ell(P') = 
\frac{\partial h(Q) \cdot u(P')}{\partial h(Q')\cdot u(P)} = \left(\frac{u(P')}{u(P)}\right)\cdot\partial\left(\frac{h(Q)}{h(Q')}\right) \in \bK^\times \partial\bF^\times = \partial\bF^\times\,. \] It follows that the coset $\ell(P)\cdot\partial\bF^\times$ does not depend on $P$. Hence neither does its $G_K$-invariant subset.

Clearly if $\tilde{\mathcal{H}}_K$ is nonempty, then it is a coset of $\tilde{\mathcal{H}}^0_K$. So it will suffice to show that \[ \left(\,\tilde{\mathcal{H}}_K \ne \emptyset\,\right) \Longrightarrow \left(\,\emptyset \ne \mathcal{H}_K \subset \tilde{\mathcal{H}}_K/K^\times\partial F^\times\,\right)\,.\] In Section \ref{InverseOfTheDescentMap} we show how to construct representatives for elements of $\Cov_0^{(p)}(C/K)$ from elements of $\tilde{\mathcal{H}}_K$, so we will assume $\mathcal{H}_K \ne \emptyset$.

To show containment, let $(D,\pi) \in \Cov_0^{(p)}(C/K)$. Its image in $\mathcal{H}_K$ is defined by a relation in the coordinate ring of $D$ of the form $\tilde{\ell}\circ\pi = \Delta\partial h$. Evaluating at any point $Q \in D$ not lying above a flex or zero of $u$, we see that $\Delta \in \ell(\pi(Q))\cdot\partial\bF^\times$. But we know $\Delta$ is Galois invariant, so it must lie in $(\ell(\pi(Q))\cdot\partial\bF^\times)^{G_K} = \tilde{\mathcal{H}}_K$.
\end{Proof}

For algebraic extensions the next lemma follows easily from the fact that $\tilde{\mathcal{H}}_K$ is defined by taking $G_K$-invariants. The value of this description is that it shows that non-membership in $\mathcal{H}_K$ is stable under base change. 

\begin{Lemma}
\label{defer}
Suppose that $K'$ is an extension of $K$ and $\Delta \in H^\times$ is such that $\Delta \otimes_K 1$ represents a class in $\mathcal{H}_{K'}$. Then the class of $\Delta$ in $H^\times/K^\times\partial F^\times$ lies in $\mathcal{H}_K$.
\end{Lemma}

\begin{proof}
The assumptions imply that in $(\bH \otimes_{\bK}\bK')^\times$ we have
\[ \frac{\Delta}{\ell(P)} \otimes 1 = \partial \alpha\,, \] for some $\alpha \in (\bF \otimes_{\bK}\bK')^\times \simeq \prod_{x\in X}\bK'^\times$ depending on $P \in C(\bK)$. For $x \in X$ we have $\Delta_x/\ell_x(P) \otimes 1 = \alpha_x^\ell$ in $\bK \otimes \bK'$. This shows that $\alpha_x$ is algebraic over $\bK$, and hence lies in $\bK$. Thus $\alpha \in \bF$ and we have $\Delta = \ell(P)\partial(\alpha)$ in $\bH$. This shows that $\Delta$ represents a class in $\mathcal{H}_K$.
\end{proof}

The elements in $\mathcal{H}_K$ satisfy certain norm conditions. These will be used in the algorithm presented in Section \ref{ComputingTheSelmerSet}.

\begin{Lemma}
\label{candbeta}
There exist $c \in K^\times$ and $\beta \in H_2^\times$ such that any $(\delta,\eps) \in F^\times\times H_2^\times \simeq H^\times$ representing a a class in $\mathcal{H}_K$ satisfies $N_{F/K}(\delta) \equiv c \Mod K^{\times p}$ and $\partial_2(\delta) = \beta\eps^p$.
\end{Lemma}

\begin{Proof}
To prove the existence of $\beta \in H_2^\times$ we argue as follows. The divisor cut out by $\tilde{\ell}_1$ is $p\x$, while that cut out by $\tilde{\ell}_2$ is $\partial_2\x$. So in the coordinate ring of $C \otimes_K H_2$ there is a relation $\partial_2(\tilde{\ell}_1) = \beta\tilde{\ell}_2^p$, which determines $\beta \in H_2^\times$. For the existence of $c \in k^\times$ note that the divisor on $C$ defined by $N_{F/k}(\tilde{\ell}_1)$ is $p$ times the divisor $\sum_{x\in X}[x]$. The divisor $\sum_{x\in X}[x]$ is itself cut out by some form $g \in K[u_1,\dots,u_p]$ of degree $p$. Thus, there is some $c \in K^\times$, such that in the coordinate ring of $C$ \[ N_{F/K}\tilde{\ell}_1 = c\cdot g^p\,.\] It is clear from the definition of $\tilde{\mathcal{H}}_K$ and Lemma \ref{tildeH} that these conditions carry over to $\mathcal{H}_K$.
\end{Proof}

\subsection{The main diagram}
\label{MainDiagram}
Let $\mathcal{K}_K$ be the finite group 
\begin{align}
\label{KappaK}
\mathcal{K}_K \simeq \frac{\HH^0\bigl(K,(\partial_2(\mu_p\bF))\bigr)}{\partial_2\left(\HH^0\left(K,\frac{\mu_p(\bF)}{\mu_p}\right)\right)} \subset \frac{\HH^0(K,\mu_p(\bar{H}_2))}{\partial_2\left(\HH^0\left(K,\frac{\mu_p(\bF)}{\mu_p}\right)\right)}\,.
\end{align}
Taking Galois cohomology of diagram (\ref{digram}) we have the following.
\begin{Proposition}[Main Diagram]
The following diagram is exact and commutative.
\[\xymatrix{ 	&& 1 \ar[d] & 1\ar[d]&\\
		1\ar[r]&\mathcal{K}_K \ar@{=}[d]\ar[r]& \mathcal{H}_K^0 \ar[r]^{\pr_1}\ar[d]& F^\times/K^\times F^{\times p}\ar[d]\ar[r]^{\partial_{2*}} & \HH^1(K,\partial_2(\mu_p(\bF))) \ar@{=}[d] \\
		1\ar[r]&\mathcal{K}_K \ar[r]&\HH^1(K,E[p])\ar[r]\ar[d]^\Upsilon& \HH^1\bigl(K,\frac{\mu_p(\bF)}{\mu_p}\bigr)\ar[d]\ar[r]^{\partial_{2*}}& \HH^1(K,\partial_2(\mu_p(\bF)))\\
		&& \Br(K)[p] \ar@{=}[r] & \Br(K)[p] & \\
}\] 
\end{Proposition}
\begin{Proof} The lower of the two rows is obtained directly from the long exact sequence of the bottom row of (\ref{digram}). Up to the identifications described below, the upper row is obtained by taking Galois cohomology of the upper row of (\ref{digram}) and then modding out by the images of $\HH^1(K,\mu_p)$. A completely formal diagram chase shows that the kernels of these two rows must be isomorphic (so that $\mathcal{K}_K$ is the kernel in the top row as well).

The identifications are the obvious ones following from Hilbert's Theorem 90, Lemma \ref{dfp} and Corollary \ref{desc2}. One can check that the map labelled $\pr_1$ is in fact induced by projection of $H \simeq F \times H_2$ onto its first factor. This is obvious from the proof of \ref{desc2}.
\end{Proof}

\begin{Remark}
For $p = 3$, one can replace $\partial_{2_*}:F^\times/K^\times F^{\times 3} \to \HH^1(K,\partial_2(\mu_p(\bF)))$ with $\partial_{2}:F^\times/K^\times F^{\times 3} \to H_2^\times /H_2^{\times p}$ without affecting exactness (see \cite[Section II.6]{CreutzThesis}). 
\end{Remark}

The relevant data for descent on $C$ is contained in a translate of the top row. For any $(\delta,\eps) \in F^\times\times H_2^\times \simeq H^\times$ representing a class in $\mathcal{H}_K$ we have a commutative diagram: 
\begin{align}
\xymatrix{ 	
&& \Cov_0^{(p)}(C/K)\ar[d]^{\tilde{\Phi}}\\
1 \ar[r] &\mathcal{K}_K\cdot(\delta,\eps)\ar[r]& \mathcal{H}_K \ar[rr]^{\pr_1} && \left\{\, \delta \in F^\times/K^\times F^{\times p}\,:\, \partial_2(\delta) \in \beta H_2^{\times p}\,\right\} \\
1\ar[r]&\mathcal{K}_K\ar[u]^{\cdot (\delta,\eps)}\ar[r]&\mathcal{H}_K^0 \ar[u]^{\cdot (\delta,\eps)} \ar[rr]^{\pr_1}&& \left\{\, \delta \in F^\times/K^\times F^{\times p}\,:\, \partial_2(\delta) \in H_2^{\times p}\,\right\}\,. \ar[u]^{\cdot\delta} }
\end{align}
Lemma \ref{candbeta} shows that the sets on the right contain $\ker(\partial_{2*})$ and its translate by $\delta$. So the lower row is an exact sequence of groups and the upper row is an exact sequence of pointed sets. 

\section{Inverse of the descent map}
\label{InverseOfTheDescentMap}
The main result of this section is the explicit construction of an inverse to the descent map. Recall that $\tilde{\mathcal{H}}_K \subset H^\times$ is the subset of elements which represent classes in the image of $\tilde{\Phi}:\Cov_0^{(p)}(C/K) \to H^\times/K^\times\partial F^\times$. 

\begin{Theorem}
\label{ConstructionWorks} 
Given $\Delta \in \tilde{\mathcal{H}}_K$, we can explicitly compute a set of $p^2(p^2-3)/2$ linearly independent quadrics over $K$ which define a genus one normal curve $D_\Delta \subset \PP^{p^2-1}$ of degree $p^2$ and a set of homogeneous polynomials defining a map $\pi_\Delta:D_\Delta \to C$ making $D_\Delta$ into a $p$-covering of $C$. Moreover, the image of the class of $(D_\Delta,\pi_\Delta)$ in $\Cov_0^{(p)}(C/K)$ under the descent map is equal to the class of $\Delta$ in $\mathcal{H}_K$.\\
\end{Theorem}

\begin{Remark}
Our proof is strongly influenced by \cite[II, Section 3]{CFOSS}. In that paper, the problem of obtaining models of $n$-coverings of elliptic curves with trivial obstruction as genus one normal curves of degree $n$ is considered. Their first step (in the `Segre embedding method') is to embed the curve as a genus one normal curve of degree $n^2$. They then show that, after projection to a suitable hyperplane, the embedding factors through the Segre embedding  $\PP^{n-1}\times (\PP^{n-1})^\vee \to \PP^{n^2-1}$. Making this factorization explicit requires an explicit trivialization of the obstruction algebra associated to the $n$-covering.

In our situation, things are actually somewhat simpler. We start with a $p$-covering of $C$. The analog of the first step of the Segre embedding method above yields a model as genus one normal curve of degree $p^2$. This already gives us what we are after. It is a feature of second $p$-descents that no trivialization of the obstruction algebra is necessary.
\end{Remark}

Let $\Delta \in H^\times$. Note that we are not yet assuming $\Delta \in \tilde{\mathcal{H}}_K$. We can associate to $\Delta$ a $C$-scheme $D_\Delta$ defined over $K$ as follows. Fix a basis $\{e_1,...,e_{p^2}\}$ for $F = \Map_K(X,\bK)$ over $K$. We can then write an arbitrary element $z \in F$ as $z = \sum_i z_ie_i$. The choice of basis gives an identification of $(F\setminus\{0\})/K^\times$ with the $K$-points of $\PP^{p^2-1}$, $0 \ne z = \sum z_ie_i$ corresponding to the point $(z_1:\dots:z_{p^2}) \in \PP^{p^2-1}$.

We start with the equation \[ \tilde{\ell} = a \Delta \partial(z) \] where $a \in K^\times$ and $z \in F\setminus\{0\}$ are treated as unknown. The map $\partial:F\to H$ can be written as a homogeneous polynomial of degree $p$ in the $z_i$ and so our equation corresponds to an equation \[ \tilde{\ell}(u_1,...,u_p) = a\Delta\partial(z_1,...,z_p),\] where $\tilde{\ell}$ and $\partial$ are homogeneous polynomials of degrees $1$ and $p$, respectively, both with coefficients in $H$. Writing this out in a basis for $H$ over $K$ (extending the basis chosen above) and equating coefficients on the basis vectors gives a system of $m := [H:K]$ equations, with coefficients in $K$, of the form:
\begin{align}
\label{eqtype}
\text{ linear in $u_1,\dots,u_p$ } = \text{ degree $p$ in $z_1,\dots,z_{p^2}$ }\,.
\end{align}

We claim that the matrix defined by the coefficients of the $m$ linear forms in the left-hand side of (\ref{eqtype}) has full rank (i.e. rank equal to $p$). For this one uses a geometric argument; over $\bK$, $\ell$ splits as a tuple of linear forms defining the hyperplanes in $Y$ and these span the space of all linear forms in $\bK[u_1,\dots,u_p]$.

Eliminating $u_1,\dots,u_p$ gives a system of equations:
\[\left\{
    \begin{array}{ccc}
      u_i=a\cdot\pi_i(z_1,\dots,z_{p^2})\,, &\text{for $i=1,\dots,p$ }\\
      0=a\cdot P_i(z_1,\dots,z_{p^2})\,, &\text{for $i=1,\dots,m-p$ } \\
    \end{array}
  \right\}
\]
where $\pi_i$ and $P_i$ are homogeneous of degree $p$ with coefficients in $K$. Let \[\{ Q_j(u_1,\dots,u_p) : j = 1,\dots,N \}\] be the homogeneous polynomials defining the model for $C$ as a genus one normal curve of degree $p$ in $\PP^{p-1}$. Recall that in the case $p=3$, $N=1$ and $Q_1$ is of degree $3$. For larger $p$, $N = \frac{p(p-3)}{2}$ and the $Q_j$ are of degree $2$. We define $D_\Delta \subset \PP^{p^2-1}(z_1:\dots:z_{p^2})$ as the (reduced) $K$-subscheme defined by the vanishing of the polynomials in the set
\[ \{ P_i :\, 0<i\le m-p \} \bigcup \{ Q_j(\pi_1,\dots,\pi_p) :\, 0 < j\le N \}\,. \] The second set is included to ensure that the rational map \[\PP^{p^2-1}(z_1:\dots:z_{p^2}) \to \PP^{p-1}(u_1:\dots:u_p)\] defined by $u_i = \pi_i(z_1,\dots,z_{p^2})$ restricts to a morphism $\pi_\Delta:D_\Delta \to C$. Note also that if $\Delta \equiv \Delta' \Mod K^\times$, then $(D_\Delta,\pi_\Delta) = (D_{\Delta'},\pi_{\Delta'})$. In other words, $(D_\Delta,\pi_\Delta)$ only depends on the class of $\Delta$ in $H^\times/K^\times$.\\

\begin{Remark}
In the case $p=3$, $m = [H:K] = 21$, so $D_\Delta \subset \PP^8$ is defined by $18$ cubics and one form of degree $9$. For $p > 3$, we have $m = p^2(p^2+1)/2$, so the model is given by $\frac{p^2(p^2+1)}{2} - p$ forms of degree $p$ and $N = \frac{p(p-3)}{2}$ forms of degree $2p$. We will see below how to obtain a set of $\frac{p^2(p^2-3)}{2}$ quadrics generating the homogeneous ideal.
\end{Remark}

One obvious, but important, property of the construction is given in the following lemma. This says that if $(D_\Delta,\pi_\Delta)$ is a $p$-covering of $C$, then its image under the descent map is necessarily given by $\Delta$.
\begin{Lemma}
\label{lifty}
If there is some $R \in C(K)$ such that $\ell(R) \in \Delta\cdot K^\times\partial F^\times$, then there exists some $Q \in D_\Delta(K)$ such that $\pi_\Delta(Q) = R$.
\end{Lemma}

\begin{Proof}
Suppose $R \in C(K)$ is such that $\ell(R) = a\Delta\partial(Q)$ with $a \in K^\times$ and $Q \in F^\times$. Choose homogeneous coordinates $(R_1:\dots:R_p)$ for $R$ and write $Q = \sum_ie_iQ_i$ with $Q_i \in K$. Recall that $\ell(R) = \frac{\tilde{\ell}(R_1,\dots,R_p)}{u(R_1,\dots,R_p)}$, where $u$ is a linear form not vanishing at $(R_1,\dots,R_p)$. Then $\tilde{\ell}(R_1,\dots,R_p) = a u(R_1,\dots,R_p) \Delta \partial(Q_1,\dots,Q_{p^2})$. Eliminating as in the construction we see that
\begin{align*}
R_i &= au(R_1,\dots,R_p)\cdot \pi_i(Q_1,\dots,Q_{p^2})\,, \text{ for $i = 1,\dots,p$,}\\
0 &= au(R_1,\dots,R_p)\cdot P_j(Q_1,\dots,Q_{p^2})\,, \text{ for $j = 1,\dots,m-p$}\,.
\end{align*} Note that $au(R_1,\dots,R_p) \in K^\times$. Since $R \in C$, the equations above  say that the point $(Q_1:\dots:Q_{p^2})$ lies in $D(K)$ and is mapped via $\pi_\Delta$ to $R$.
\end{Proof} 

The association $H^\times \ni \Delta \mapsto (D_\Delta,\pi_\Delta)$ depends on the choice of basis for $F$ over $K$. We assume all $(D_\Delta,\pi_\Delta)$ are constructed using the same basis (and so live in the same copy of $\PP^{p^2-1}$). A different choice of basis leads to objects which differ only by a linear automorphism of the ambient space. It is to be understood that this automorphism is applied to each $(D_\Delta,\pi_\Delta)$ if we change the basis.

When working geometrically, it will be convenient to use the basis of $\bF$ over $\bK$ given by the characteristic functions. These are the maps $e_x \in \bF = \Map(X,\bK)$ (indexed by $x \in X)$ taking the value $1$ at $x$ and the value $0$ at all $x' \ne x$. In terms of this basis, $0 \ne z \in \bF \setminus\{0\}$ corresponds to the point $z = (z_x) \in \PP^{p^2-1}$ with $z_x$-coordinate given by the value of $z$ at $x$. We can extend to a basis for $\bH$ over $\bK$ by taking the characteristic functions on $Y$ and identifying $x \in X$ with the hyperplane in $Y$ cutting out the divisor $p[x]$ on $C$. Then $\partial(z)$ splits as the tuple of polynomials, (indexed by $y \in Y$) \[ \partial(z) = \left( \prod_{x\in y} z_x \right)_{y \in Y}\,,\] where as usual the product is to be taken with appropriate multiplicities.

\begin{Lemma}
\label{ComputeQuadrics}
Given $\Delta \in H^\times$ we can explicitly compute a set of $p^2(p^2-3)/2$ linearly independent quadrics over $K$ which lie in the homogeneous ideal of $D_\Delta \subset \PP^{p^2-1}$.
\end{Lemma}

\begin{Remark}
We are not (yet) claiming that these quadrics define $D_\Delta$; we have also not assumed that $\Delta \in \tilde{\mathcal{H}}_K$.
\end{Remark}

\begin{Proof}
Under the splitting $H \simeq F \times H_2$, write $\Delta = (\Delta_1,\Delta_2)$. The equation $\tilde{\ell} = a\Delta\partial(z)$ corresponds to the two equations 
\begin{align}
\label{2eqs}
\tilde{\ell}_1 = a\Delta_1 z^p \text{ and } \tilde{\ell}_2 = a\Delta_2 \partial_2(z)\,.
\end{align}

First consider the case $p \ge 5$. Recall that, as a $G_K$-set $Y_2 \simeq X \times \frac{E[p]\setminus\{0_E\}}{\{\pm 1\}}$ and that $F$ may be viewed as a subalgebra of $H_2$. The hyperplanes in $Y_2$ cut out divisors on $C$ of the form $(p-2)[x] + [x+P] + [x-P]$. So there is a quadratic form $\tilde{N}$ such that $\partial_2(z) = z^{p-2}\tilde{N}(z)$. We can obtain a homogeneous equation in $H_2$ by taking the ratio of the two equations in (\ref{2eqs}) and multiplying through by $z^2$. We get
\begin{align}
\label{1eqs}
\frac{\tilde{\ell}_2}{\tilde{\ell}_1}\cdot z^2 = \left(\frac{\Delta_2}{\Delta_1}\right)\cdot\tilde{N}(z)\,.
\end{align}

To achieve the same when $p=3$, recall that $F$ corresponds to the $G_K$-set $X$ consisting of the $9$ flex points while $H_2$ corresponds to the $G_K$-set $Y_2$ consisting of the $12$ lines in $\PP^2$ passing through three distinct flex points. We can no longer view $F$ as a subalgebra of $H_2$. Instead we work with the \'etale algebra $M = \Map_K(Z,\bK)$ associated to the $G_K$-set $Z$ consisting of all pairs $(x,y) \in X\times Y_2$ such that $x \in y$. Each flex is contained in four lines, while each line passes through three flexes so \[ [M:F] = 4 \textnormal{ and } [M:H_2] = 3\,.\] The induced norm  \[\partial=\partial_1\times\partial_2:F \to F\times H_2\] is given by $\partial_1(z) = z^3$ and $\partial_2(z) = N_{M/H_2}(z)$. So, identifying $z$ with its image in $M$, we can write $\partial_2(z) = z\tilde{N}(z)$ for some quadratic form $\tilde{N}$. Over $M$ we can write our equations as 
\[ \tilde{\ell}_1 = a\Delta_1 z^3 \text{ and } \tilde{\ell}_2 = a\Delta_2 z \tilde{N}(z)\,.\] Again we can obtain a homogeneous equation in $M$ by taking the ratio. We get an equation
\[ \frac{\tilde{\ell}_2}{\tilde{\ell}_1}\cdot z^2 = \left(\frac{\Delta_2}{\Delta_1}\right)\cdot \tilde{N}(z)\,.\] 

Formally this is exactly what was obtained for $p \ge 5$. Note also that $[M:K] = 3^2(3^2-1)/2$, while for larger $p$ we have $[H_2:K] = p^2(p^2-1)/2$. So in either case, writing the equation out in terms of the basis over $K$ gives $p^2(p^2-1)/2$ quadrics some of whose coefficients are rational functions on $C$. These can be eliminated using linear algebra over $K$ to obtain a set of quadrics with coefficients in $K$ which vanish on $D_\Delta$.

We want to count the number of independent quadrics left after eliminating. For this we may work geometrically. For $p \ge 5$ we can index the elements of $Y$ by pairs $(x,P) \in X \times \frac{E[p]}{\{\pm 1\}}$. The linear form $\tilde{\ell}$ splits over $\bK$ as $\tilde{\ell} = (\tilde{\ell}_{(x,P)})$, where $\tilde{\ell}_{(x,P)}$ is a linear form with coefficients in $\bK$ defining the hyperplane whose intersection with $C$ is given by the divisor $(p-2)[x]+[x+P]+[x-P]$. Note that $P=0_E$ is allowed. For $p=3$ we can do the same, but with the caveat that the indexing is no longer unique. Namely, each line $y \in Y_2$ corresponds to three pairs $(x,P) \in X \times \frac{E[p]}{\{\pm 1\}}$ (we get one pair for each $x \in y$). In any case, we can still use the index $_{(x,P)}$ to denote the factor of $\bH$ corresponding to the line in $\PP^2$ whose intersection with $C$ is given by the divisor $[x] + [x+P] + [x-P]$.

The notation is such that for distinct $(x,P) \in X \times \frac{E[p]\setminus\{0\}}{\{\pm 1\}}$, we have distinct rational functions \[ G_{(x,P)} := \frac{\tilde{\ell}_{(x,P)}}{\tilde{\ell}_{(x,0)}}\in\kappa(\bC)^\times\] with divisors $\diw(G_{(x,P)}) = [x+P]+[x-P]-2[x]$. Over $\bK$, we can work with the basis of $\bF$ given by the characteristic functions and use $(z_x)$ for coordinates on $\PP^{p^2-1}$. In terms of these and the $G_{(x,P)}$ the homogeneous equation (\ref{1eqs}) corresponds to a system of equations
\begin{align}
\label{gxp}
G_{(x,P)}\cdot z_x^2 = \tilde{\Delta}_{(x,P)}\cdot z_{x+P}z_{x-P}\,,
\end{align} parameterized by $(x,P) \in \frac{E[p]\setminus\{0\}}{\{\pm 1\}}$, where for simplicity we have denoted $\Delta_{(x,P)}/\Delta_{(x,0)}$ by $\tilde{\Delta}_{(x,P)}$.

For fixed $x$, the $(p^2+1)/2$ linear forms $\tilde{\ell}_{(x,P)}$, with $P \in E[p]/\{\pm1\}$, all define hyperplanes meeting $C$ in $x$ with multiplicity at least $p-2$. This gives $p-2$ nontrivial relations among them. The matrix given by the coefficients of the $\tilde{\ell}_{(x,P)}(u_1,\dots,u_p)$ has rank $\le p-(p-2) = 2$. On the other hand, the rank must be greater than one since these do not all define the same hyperplane. This introduces a dependence among the $G_{(x,P)}$. Alternatively one can argue that the functions $G_{(x,P)}$ are all in the Riemann-Roch space $\mathcal{L}(2[x])$ which has dimension $2$. 

In any event, if we fix $P_0 \in E[p]\setminus\{0\}$, then for any $P\in \frac{E[p]\setminus\{0\}}{\{\pm 1\}}$, we can find $a_{P},b_{P} \in \bK$ such that \[ G_{(x,P)} = a_{P}G_{(x,P_0)} + b_{P}\,.\] Using this to eliminate the $G_{(x,P)}$ from (\ref{gxp}) we obtain a set of quadrics
\[ a_{P}\tilde{\Delta}_{(x,P_0)}\cdot z_{x+P_0}z_{x-P_0} + b_{P}\cdot z_x^2 = \tilde{\Delta}_{(x,P)}\cdot z_{x+P}z_{x-P}\,,\] parameterized by $P \in \frac{E[p] \setminus\{0,\pm P_0\}}{\{\pm1\}}$ and with coefficients in $\bK$. Since $\tilde{\Delta}_{(x,P)} \ne 0$, these are necessarily independent. Note also that the monomials appearing in these quadrics are all of the form $z_{x+Q}z_{x-Q}$ for some $Q \in E[p]/\{\pm 1\}$. A different choice for $x$ leads to quadrics involving a disjoint set of monomials. So, in total this gives a set of $\#X\cdot\#\left(\frac{E[p] \setminus\{0,\pm P_0\}}{\{\pm1\}}\right) = p^2(p^2-3)/2$ independent quadrics as required.
\end{Proof}

There is an obvious action of $\bF^\times$ on $(\bF\setminus\{0\})/\bK^\times$ by multiplication. The choice of basis gives an identification of the latter with the $\bK$-points of $\PP^{p^2-1}$ and hence a representation \[ \bF^\times \ni \alpha \mapsto \varphi_\alpha \in \PGL_{p^2} = \Aut(\PP^{p^2-1})\,.\] Working with the basis of $\bF$ given by the characteristic functions, the representation takes the particularly simple form $\alpha = (\alpha_x) \mapsto \text{Diagonal}(\alpha_x)$; this is just coordinate-wise multiplication. Assuming we are working with a basis for $F$ over $K$, we see that for any extension of fields $K'/K$, \[\varphi_\alpha \in \PGL_{p^2}(K') \Leftrightarrow \alpha \in (F\otimes_KK')^\times\,.\]

\begin{Lemma}
\label{twistDelta}
For any $\Delta \in \bH^\times$ and $\alpha \in \bF^\times$, the action of $\alpha$ on $\PP^{p^2-1}$ induces an isomorphism (of $C$-schemes) $\varphi_\alpha:D_{\partial(\alpha)\cdot\Delta} \to D_\Delta$.
\end{Lemma}

\begin{Corollary}
\label{twistDeltaK}
Let $\Delta \in H^\times$ and $(D_\Delta,\pi_\Delta)$ be the corresponding $C$-scheme. The coset $\Delta\mathcal{H}^0_K \subset H^\times/K^\times\partial F^\times$ parameterizes a set of twists of $(D_\Delta,\pi_\Delta)$ as a $C$-scheme defined over $K$ up to $K$-isomorphism. 
\end{Corollary}

\begin{Proof} To prove the lemma, use that $D_{\partial(\alpha)\cdot\Delta}$ is defined by the equation $\tilde{\ell} = \partial(\alpha)\Delta\partial(z)$. If $Q \in D_{\partial(\alpha)\cdot\Delta}$ is any point mapping to, say $P \in C$, then the point $\alpha Q \in \PP^{p^2-1}$ evidently satisfies \[\Delta \partial(\alpha Q) \equiv \Delta \partial(\alpha)\partial(Q) \equiv \tilde{\ell}(P) \Mod K^\times\,.\] The equivalence here is meant for any choices of coordinates for $P$ and $Q$. This means $\alpha Q$ is a point of $D_\Delta$ lying above $P$. This proves the lemma.

The lemma implies that if $\Delta \in H^\times$, then the $C$-schemes corresponding to the elements of $\Delta\tilde{\mathcal{H}}^0_K = \Delta(\partial{\bF^\times})^{G_K}$ are all twists of $(D_\Delta,\pi_\Delta)$. The isomorphism $\varphi_\alpha:D_{\partial(\alpha)\Delta} \to D_\Delta$ is defined over $K$ if and only if $\alpha \in F^\times$ in which case $\partial(\alpha)\Delta$ and $\Delta$ differ by an element of $K^\times\partial F^\times$. So $\Delta \mathcal{H}^0_K$ parameterizes the corresponding twists in $\Delta\tilde{\mathcal{H}}_K^0$ up to $K$-isomorphism.
\end{Proof}

By definition, any twist of a $p$-covering is a $p$-covering, so we can reduce to the geometric situation. To prove the theorem, it is enough to show that there is some $\Delta \in \tilde{\mathcal{H}}_{\bK}$ such that $(D_{\Delta},\pi_{\Delta})$ is a $p$-covering of $C$ defined over $\bK$. The proof of the following lemma also shows that for $\Delta \in \tilde{\mathcal{H}}_K$, the $p^2(p^2-3)/2$ quadrics obtained in Lemma \ref{ComputeQuadrics} generate the homogenous ideal of $D_\Delta$.

\begin{Lemma}
\label{EmbeddC}
There exists some $\Delta \in \tilde{\mathcal{H}}_{\bK}$ such that $(D_{\Delta},\pi_{\Delta})$ is a $p$-covering of $C$.
\end{Lemma}

\begin{Proof}
For this we may work over $\bK$, using the basis given by the characteristic functions and $z_x$ for coordinates on $\PP^{p^2-1}$. Choosing any flex point $x_0 \in X$ as origin, we may consider $(C,x_0)$ as an elliptic curve over $\bK$. Denote the multiplication by $p$ map on $(C,x_0)$ by $\pi:C \to C$. This is a $p$-covering of $C$. We are going to find some $\Delta \in \bH^\times$ representing the image of $(C,\pi)$ under the descent map and then show that the scheme $D_\Delta$ produced by the construction above is equal to the image of $(C,\pi)$ under a certain embedding into $\PP^{p^2-1}$ as a genus one normal curve of degree $p^2$.

To compute the image of $(C,\pi)$ under the descent map we use the definition. Namely, we embed $C$ in $\PP^{p^2-1}$ in such a way that the pull-back of any flex point is a hyperplane section. This amounts to finding a basis for the Riemann-Roch space of the divisor $\pi^*[x_0]$. For each $x\in X$, we can find a rational function $G_x \in \kappa(\bC)^\times$ with divisor $\diw(G_x) = \pi^*[x] - \pi^*[x_0]$. For existence of these functions, note that $\pi$ is multiplication by $p$ on the elliptic curve $(C,x_0)$ and recall that the Weil pairing on $(C,x_0)$ is defined in terms of such functions (see \cite[III.8]{Silverman}). By Riemann-Roch the dimension of $\mathcal{L}(\pi^*[x_0])$ is $p^2 = \#X$. Clearly the $G_x$ lie in the Riemann-Roch space, so it will suffice to show that they are linearly independent. This follows from the definition of the Weil pairing; the $G_x$ are eigenfunctions for distinct characters with respect to the action of $X=C[p]$ by translation. (see the first paragraph of the proof of \cite[II, Proposition 3.3]{CFOSS}).

Thus we may define an embedding of $C$ into $\PP^{p^2-1}$ via 
\[ g:C\ni P \mapsto (G_x(P))_{x\in X} \in \PP^{p^2-1}\,.\] It is evident that the pull-back of any flex point $x \in X$ by $\pi$ is the hyperplane section of $g(C) \subset \PP^{p^2-1}$ cut out by $z_x = 0$. Let $Q \in C \setminus C[p^2]$ be any point, with projective coordinates $g(Q)$. By the definition of the descent map, the image of $(C,\pi)$ under the descent map is represented by the $\Delta \in \bH^\times$ such that
\begin{align}
\label{anothereq}
\tilde{\ell}(\pi(Q)) = \Delta\partial(g(Q))\,.
\end{align} By definition we have that $\Delta \in \tilde{\mathcal{H}}_{\bK}$.

Equation (\ref{anothereq}) was also used to construct $D_\Delta$. So it is clear that $\pi_\Delta \circ g = \pi$ on $C\setminus C[p^2]$ and that the image of this open subscheme under $g$ is contained in $D_\Delta$. Since $D_\Delta$ is projective (hence complete), this is then true on all of $C$. We conclude that $g(C) \subset D_\Delta$ and that $\pi_\Delta\circ g = \pi$. On the other hand, $g(C)$ is a genus one normal curve of degree $p^2$. Its homogeneous ideal can be generated by a $\bK$-vector space of quadrics of dimension $p^2(p^2-3)/2$. We have already found a set of $p^2(p^2-3)/2$ linearly independent quadrics vanishing on $D_\Delta$ in Lemma \ref{ComputeQuadrics}, so we must have $g(C) = D_\Delta$. Thus $(D_\Delta,\pi_\Delta)$ is a twist of $(C,\pi)$. This completes the proof.
\end{Proof}

\section{Computing the $p$-Selmer Set}
\label{ComputingTheSelmerSet}
We shift our focus now to the arithmetic situation. We specialize to the case that $K=k$ is a number field. We assume that $C$ is an everywhere locally solvable genus one normal curve of degree $p$ defined over $k$. Recall that local solvability implies that $\Pic(C) = \Pic(\bC)^{G_k}$ and that $\Cov^{(p)}(C/k) \ne \emptyset$. Thus all of the material of Sections \ref{AffineStructure} -- \ref{InverseOfTheDescentMap} applies to $C$.

An element in an \'etale $k$-algebra $A \simeq \prod_i K_i$ will be called {\em integral} if its image in each $K_i$ is integral. We assume that the linear form $\tilde{\ell}$ defining the descent map and all polynomials appearing in the model for $C$ have integral coefficients. We further assume that the constants $c \in k^\times$ and $\beta \in H_2^\times$ given by \ref{candbeta} are integral. All of this can be achieved by scaling. We denote the completion of $k$ at a prime $v$ by $k_v$. We attach a subscript $v$ to any object defined over $k$ to denote the corresponding object over $k_v$ obtained by extension of scalars. For example $H_v=H\otimes k_v$, $\tilde{\mathcal{H}}_v = \tilde{\mathcal{H}}_{k_v}$, $C_v = C \otimes k_v$, and so on.

\subsection{The algebraic Selmer set}
\label{AlgebraicSelmerSet}
The descent map allows us to identify $\Sel^{(p)}(C/k)$ with its image in $H^\times/k^\times\partial F^\times$. We now determine the image. This gives an algebraic presentation of the $p$-Selmer set, which can be computed fairly directly.\\

Consider the following diagram:
\[ \xymatrix{	\Pic(C) \ar[rr]^\Phi\ar[d] && H^\times/k^\times\partial F^\times \ar[d]^{\prod_v\res_v} \\
		\prod_v\Pic(C_v) \ar[rr]^{\prod_v\Phi_v} && \prod_v H_v^\times/k_v^\times\partial F_v^\times\,. \\ } \]
If $(D,\pi) \in \Sel^{(p)}(C/k)$ is an everywhere locally solvable $p$-covering of $C$, then its image, $\tilde{\Phi}((D,\pi)) \in H^\times/k^\times\partial F^\times$, has the property that it maps under $\prod_v\res_v$ into the subset $\prod_v \Phi_v(\Pic^1(C_v)) \subset \prod_v \mathcal{H}_v$. This suggests the following definition.

\begin{Definition}
The {\em algebraic $p$-Selmer set of $C$ associated to $\Phi$} is the set
\[ \Sel_{alg}^{(p)}(C/k) = \{ \Delta \in H^\times/k^\times\partial F^\times\,:\,\res_v(\Delta) \in \Phi_v(\Pic^1(C_v)) \text{ for all $v$ }\}\,.\]
\end{Definition}

\begin{Theorem}
The descent map gives a one to one correspondence between the $p$-Selmer set of $C$ and the algebraic $p$-Selmer set of $C$.
\end{Theorem}

\begin{Proof}
\label{SELALG}
The defining property of the descent map shows that the image of $\Sel^{(p)}(C/k)$ is equal to $\Sel^{(p)}_{alg}(C/k) \cap \mathcal{H}_k$. We know that the descent map is injective by Proposition \ref{Injective}, so it suffices to show that $\Sel^{(p)}_{alg}(C/k) \subset \mathcal{H}_k$. This follows from Corollary \ref{defer};  $\Phi_v(\Pic^1(C_v)) \subset \mathcal{H}_v$ and any element of $H^\times/k^\times\partial F^\times$ which restricts into $\mathcal{H}_v$ (for some $v$) is an element of $\mathcal{H}_k$.
\end{Proof}

One can formulate the same definition for divisor classes of degree $0$.
\begin{Definition}
The {\em algebraic $p$-Selmer group of $E = \Jac(C)$} is
\[ \Sel_{alg}^{(p)}(E/k) = \{ \Delta \in H^\times/k^\times\partial F^\times\,:\,\res_v(\Delta) \in \Phi_v(\Pic^0(C_v)) \text{ for all $v$ }\}\,.\]
\end{Definition}

Note that by \ref{defer}, $\Sel_{alg}^{(p)}(E/k) \subset \mathcal{H}^0_k$. Since $\mathcal{H}_k$ is a principal homogeneous space for $\mathcal{H}^0_k$, the same is true of the corresponding Selmer objects.
\begin{Lemma}
\label{FakeCoset}
If the algebraic $p$-Selmer set of $C$ is nonempty, then it is a coset of the algebraic $p$-Selmer group of $E$ inside $H^\times/k^\times\partial F^\times$.
\end{Lemma}

\begin{Proof}
By assumption $C$ is everywhere locally solvable. So, everywhere locally the group of $k_v$-rational divisor classes of degree $1$ on $C$ is a coset of the group of $k_v$-rational divisor classes of degree $0$. Since $\Phi_v$ is a homomorphism, the same is true of their images in $H_v^\times/k_v^\times\partial F_v^\times$. If the algebraic $p$-Selmer set of $C$ is nonempty, then these cosets can be simultaneously defined by some global element of $\mathcal{H}_k$.
\end{Proof}

Although it is not reflected in the notation, $\Sel_{alg}^{(p)}(C/k)$ depends on our choice of linear form $\tilde{\ell}$ used to define the descent map and the algebraic $p$-Selmer group of $E$ depends on $C$. The next proposition shows, however, that the image of $\Sel_{alg}^{(p)}(E/k)$ in $\HH^1(k,E[p])$ does not.

\begin{Proposition}
The inclusion $\mathcal{H}^0_k \simeq C^\perp \hookrightarrow \HH^1(k,E[p])$ identifies the algebraic $p$-Selmer group of $E$ with the $p$-Selmer group of $E$.
\end{Proposition}

\begin{Proof}
We identify $\mathcal{H}^0_k$ with its image in $\HH^1(k,E[p])$ and $E$ with $\Pic^0(C)$. To show that the algebraic Selmer group is contained in the Selmer group we use Lemma \ref{DescentOnE}. This says that the images of $\Phi_v|_{\Pic^0_v(C)}$ and the connecting homomorphism $\delta_v$ from the Kummer sequence of $E/k_v$ are the same. So clearly the algebraic Selmer group is contained in the Selmer group.

For the reverse inclusion it suffices to show that $\Sel^{(p)}(E/k) \subset \mathcal{H}^0_k \simeq C^\perp$. So we need to show that elements of the Selmer group are orthogonal to $C$ with respect to the Weil pairing induced cup product of level $p$. Using that the cup product is the bilinear form associated to the obstruction map we have \[ C \cup_pC' = \Ob_p(C+C') - \Ob_p(C) - \Ob_p(C')\,,\] for any $C' \in \HH^1(k,E[p])$. If $C'$ is everywhere locally solvable then so is $C+C'$ (because the Selmer group is a group). Having points everywhere locally implies trivial obstruction, so all the terms on the right-hand side vanish as required.
\end{Proof}

\subsection{Computable description} In order to compute the algebraic Selmer set explicitly, we need a method for determining these local images. For a given $v$, this is relatively straight forward, but there are infinitely many primes to deal with. We will show that, for all but finitely many primes, the local image can be identified with the unramified subgroup. This feature, which is typical of explicit descents, allows us to apply classical (and effective) finiteness theorems in algebraic number theory to reduce the problem to a finite computation. We note also that since $p$ is assumed to be odd we can ignore all archimedean primes.

For a completion $k_v$ of $k$ at a non-archimedean prime, we use $k^{\unr}_v$ to denote the maximal unramified extension of $k_v$. If $\xi$ is an element of some object defined over $k$, we say that $\xi$ is unramified at $v$ if $\xi$ becomes trivial upon extension of scalars to $k_v^{\unr}$. For Galois cohomology groups $\HH^1(k,-)$, this coincides with the usual definition that $\xi$ be in the kernel of the restriction map to $\HH^1(k_v^{\unr},-)$. For example, a class in $F^\times/k^\times F^{\times p}$ represented by $\delta$ is unramified at $v$ if $\delta \in k_v^{\unr\times} (F\otimes k_v^{\unr})^{\times p}$ or, equivalently if its image under the map $F^\times/k^\times F^{\times p} \hookrightarrow \HH^1(k,\mu_p(\bF)/\mu_p) \to \HH^1(k_v^{\unr},\mu_p(\bar{F_v})/\mu_p)$ is zero. For a finite set of primes $S$ and a $k$-Algebra $A$ we use $A(S,p)$ to denote the the finite group of elements of $A^\times/A^{\times p}$ which are unramified outside $S$.

The first step is to identify a suitable finite set of bad primes. To that end, let $F'$ denote the field extension of $k$ obtained by adjoining the coordinates of all flex points of $C$. We refer to $F'$ as the {\em splitting field of $X$}. We can write the linear form used to define the descent map as $\tilde{\ell} = (\tilde{\ell}_1,\tilde{\ell}_2)$ under the splitting $H \simeq F\times H_2$. Here $\tilde{\ell}_1$ defines a hyperplane section meeting $C$ at a generic flex point with multiplicity $p$. Over $F'$, all flex points are defined, and so $\tilde{\ell}_1$ splits as a $p^2$-tuple, $(\tilde{\ell}_x)_{x\in X}$ of linear forms with coefficients in $F'$ each defining the hyperplane meeting $C$ only at the flex $x \in X$. 

At any non-archimedean prime $w$ of $F'$, we can reduce the $\tilde{\ell}_x \Mod w$. Since this linear form may vanish $\Mod w$, it may fail to define a hyperplane section of the reduction of $C \Mod w$. In some sense this is a situation we would like to avoid. We will refer to a prime $v$ of $k$ as a prime of bad reduction (resp. good reduction) for $\tilde{\ell}$ if there is some (resp. no) prime $w|v$ of $F'$ for which this occurs.

Recall the constant $c \in k^\times$ defined in Lemma \ref{candbeta}. By scaling we may assume $c$ to be integral.

\begin{Lemma}
\label{Unramifiedv}
Let $v$ be a non-archimedean prime of $k$ which is of good reduction for both $C$ and $\tilde{\ell}$ and which is prime to both $p$ and $c$. Then $\Phi_v(\Pic^1_v(C)) \subset H_v^\times/k_v^\times\partial F_v^\times$ is contained in the unramified subgroup.
\end{Lemma}

\begin{Proof}
Let $F'$ be the splitting field of $X$. By the criterion of Neron-Ogg-Shafarevich, the primes which ramify in the extension $F'/k$ are either primes of bad reduction for $C$ or lie above $p$. In particular, if $v$ is as in the statement, then it does not ramify in $F'$.

Now we claim that if $v$ does not ramify in $F'$, then for all $\Delta \in \mathcal{H}_v$ we have 
\[ \Delta \in \mathcal{H}_v \text{ is unramified } \Longleftrightarrow \pr_1(\Delta) \in F_v^\times/K_v^\times F_v^{\times p} \text{ is unramified.} \]
This follows from the fact that for these `good primes' the map \[ \mathcal{H}_{k_v^{\unr}} \to F^{\unr\times}/K^{\unr\times}F^{\unr\times p}\] induced by projection onto the first factor of $H \simeq F\times H_2$ is injective. To see the injectivity recall that the fibers of this map (see the diagrams in \ref{MainDiagram}) are parameterized by \[\mathcal{K}_v := \frac{\HH^0\bigl(k_v^{\unr},(\partial_2(\mu_p\bF))\bigr)}{\partial_2\left(\HH^0\left(k_v^{\unr},\frac{\mu_p(\bF)}{\mu_p}\right)\right)}\,.\] As $v$ does not ramify in $F'$, all flex points are defined over $k_v^\unr$. So the actions on the modules appearing here are trivial. Since $\mu_p \subset \ker\partial_2$, we have $\partial_2(\mu_p(\bF)) = \partial_2(\mu_p(\bF)/\mu_p)$. So the quotient is trivial.

To prove the lemma, it now suffices to show that the image of the composition \[ \Pic^1(C_v) \stackrel{\Phi_v}{\To} \mathcal{H}_v \stackrel{\pr_1}{\To} F_v^\times/k_v^\times F_v^{\times p}\] is unramified. For this it will be enough to show that this is true of any point $P \in C(k_v)$ which is neither a zero nor a pole of $\ell_1$. For this we can choose primitive integral coordinates for $P$ (i.e. homogeneous coordinates with valuations that are non-negative but not all positive) and consider $\tilde{\ell}_1(P) \in (F\otimes k_v)^\times$. The flex algebra $F\otimes k_v$ splits as a product of extensions of $k_v$. Since $v \nmid p$, in order that the image of $P$ be unramified it is sufficient that the valuation of $\tilde{\ell}_1(P)$ in each of these factors is a multiple of $p$.

Fix some factor $K_{\mathfrak{v}}$. For any prime $\mathfrak{w}$ of $F'$ extending $\mathfrak{v}$, we get an unramified tower of fields $k_v \subset K_{\mathfrak{v}} \subset F'_{\mathfrak{w}}$. Let $\nu_{\mathfrak{w}}$ be the normalized valuation on $F'_{\mathfrak{w}}$. Over $F'$, $\tilde{\ell}_1$ splits as $(\tilde{\ell}_x)_{x\in X}$ and, since the extensions are all unramified, it suffices to show that $\nu_\mathfrak{w}(\tilde{\ell}_x(P)) \equiv 0 \Mod p$ for each $x \in X$. 

For this we make use of the norm condition. Since $v \nmid c$, its valuation satisfies the congruence $\nu_{\mathfrak{w}}(c)\equiv 0 \Mod p$. Hence, \[ \sum_{x\in X} \nu_\mathfrak{w}(\tilde{\ell}_x(P)) = \nu_{\mathfrak{w}}\left(\prod_{x \in X} \tilde{\ell}_x(P) \right) \equiv \nu_{\mathfrak{w}}(c) \equiv 0 \Mod p\,.\] Each summand appearing on the left is nonnegative. To complete the proof it suffices to show that at most one can be positive.

Since $v$ is of good reduction for $\tilde{\ell}$, the reduction of each $\tilde{\ell}_x$ defines a hyperplane meeting $\tilde{C}$ only at the image $\tilde{x}$ of $x$ on $\tilde{C}$. So $\nu_{\mathfrak{w}}(\tilde{\ell}_x(P)) > 0$ if and only if $P$ and $x$ have the same image under the reduction map. On the other hand, the images of the flex points modulo $\mathfrak{w}$ are all distinct since $v$ is of good reduction for $C$ and is prime to $p$ (the images of these flex points are the flex points of the reduced curve). So $\nu_{\mathfrak{w}}(\tilde{\ell}_x(P))$ can be positive for at most one $x\in X$. This completes the proof.
\end{Proof}

\begin{Proposition}
\label{Selalg}
Let $S$ be the set of primes of $k$ containing all non-archimedean primes dividing $p$ or $c$, all primes of bad reduction for $C$ or $\tilde{\ell}$. Let $\mathcal{H}_S$ denote the subgroup of $\mathcal{H}_k$ consisting of elements that are unramified outside $S$. Then 
\[ \Sel_{alg}^{(p)}(C/k) = \{ \Delta \in \mathcal{H}_S\,:\,\res_v(\Delta) \in \Phi_v(\Pic^1(C_v)), \text{ for all $v \in S$ }\}\,.\]
\end{Proposition}

\begin{Proof}
Let $Z$ denote the set in the statement. The previous lemma shows that $Z$ contains $\Sel_{alg}^{(p)}(C/k)$. To show that the reverse inclusion holds, we may assume that $Z$ is nonempty. Let $\Delta \in Z$ and $v \notin S$. To show that $\Delta \in \Sel_{alg}^{(p)}(C/k)$ we must show that $\res_v(\Delta) \in \Phi_v(\Pic^1(C_v))$. Choose any $P \in \Pic^1_v(C)$. Then both $\Phi_v(P)$ and $\res_v(\Delta)$ are unramified, so $\Phi_v(P)\cdot\res_v(\Delta)^{-1}$ is in the unramified subgroup of $\mathcal{H}_v^0$.

Since $v$ is a prime of good reduction for $C$, it is also a prime of good reduction for its Jacobian. For primes outside $S$, the image of the connecting homomorphism $\delta_v:E(k_v) \to \HH^1(k_v,E[p])$ is equal to the unramified subgroup. On the other hand, \ref{DescentOnE} says that $\Phi_v:\Pic^0(C_v) \to \mathcal{H}_v^0 \subset \HH^1(k_v,E[p])$ coincides with connecting homomorphism. It follows that $\Phi_v(\Pic^0(C_v))$ is equal to the unramified subgroup of $\mathcal{H}_v^0$. Hence there exists some $Q \in \Pic^0_v(C)$ such that $\Phi_v(Q) = \Phi_v(P)\res_v(\Delta)^{-1}$. Since $\Phi_v$ is a homomorphism we have $\res_v(\Delta) = \Phi_v(P - Q)$, which completes the proof since $P-Q \in \Pic^1_v(C)$.
\end{Proof}

\subsection{The Algorithm}
\label{Algorithm}
The theory above gives rise to the following algorithm for computing a set of representatives for the algebraic $p$-Selmer set of $C$. The output is a collection of elements in $H^\times$. Using the methods of Section \ref{InverseOfTheDescentMap}, these can then be turned quite easily into explicit models as genus one normal curves of degree $p^2$. Thus we have an algorithm for performing explicit second $p$-descents. For $p=3$ and $k = \Q$, our implementation of this algorithm has been contributed to \magma\cite{magma}. For practical applications it is also important to find `nice models' (e.g. one with small coefficients -- see \cite{MinRed} for details). With the help of Tom Fisher and Michael Stoll we have implemented some ad hoc methods. However, there is much room for both theoretical and practical improvement. 

For larger $p \ge 5$ the algorithm is currently impractical for two reasons. The first of these is the, largely unavoidable, computation of $S$-class and -unit group information in $F$. With the current state of the art, this becomes somewhat prohibitive already for $p = 5$. There is hope, however, that this will become feasible for larger $p$ in the near future as computing power and algorithms in algebraic number theory improve. The second arises from the fact that the degree of the algebra $H_2$ is simply too large. Generically it is a number field of degree $p^2(p^2-1)/2$ over $k$. The algorithm does not, however, require class and unit group information in $H_2$. The most expensive operations required are the extraction of $p$-th roots. Even so, this quickly becomes impractical.

{\bf Compute $\Sel_{alg}^{(p)}(C/k)$:}
\begin{enumerate}
\item Compute the algebras $F$ and $H_2$, the map $\partial_2:F\to H_2$, the linear form $\tilde{\ell}$, the constants $c \in k^\times$, $\beta \in H_2^\times$, and the set $S$ of bad primes.
\item Let $V_1 \subset F^\times$ be a (finite) set of representatives for the unramified outside $S$ subgroup of $F^\times/k^\times F^{\times p}$.
\item Let $V_2 = \{ \delta \in V_1\,:\,N_{F/k}(\delta) \equiv c \Mod \Q^{\times p} \}\,.$
\item For each $v \in S$, determine the local image $\Phi(\Pic^1(C_v)) \subset \mathcal{H}_v$.
\item Let $V_3 = \{ \delta \in V_2\,:\, \forall\,v\in S,\,\res_v(\delta) \in \pr_1(\Phi(\Pic^1(C_v)))\,\}\,.$
\item Let $V_4$ be the set of $(\delta,\eps) \in F^\times\times H_2^\times$ such that $\delta \in V_3$ and  $\eps \in H_2^\times$ is a $p$-th root of $\partial_2(\delta)/\beta$, modulo the equivalence $(\delta,\eps)\sim(\delta,\eps')\,\Leftrightarrow\,\eps/\eps' \in \partial_2((\mu_p(\bF)/\mu_p)^{G_k})$.
\item Let $V_5 = \{ (\delta,\eps) \in V_4\,:\, \forall\,v\in S,\,\res_v(\delta,\eps) \in \Phi(\Pic^1(C_v))\,\}\,.$
\item return $V_5$.
\end{enumerate}

\begin{Remark}
The reason for including steps (3) and (5) is to reduce the size of $V_1$ as much as possible before proceeding to step (6) where one has to extract $p$-th roots.
\end{Remark}

Let us prove that the algorithm returns a set of representatives for the algebraic Selmer set. The equivalence in step (6) is included to ensure that the $(\delta,\eps) \in V_5$ represent distinct classes modulo $k^\times\partial F^\times$. In the proof of Lemma \ref{Unramifiedv} we have seen that, for primes not in $S$, a class in $H^\times/k^\times \partial F^\times$ is unramified if and only if its image in $F^\times/k^\times F^{\times p}$ is unramified. Moreover the elements of $V_5$ are in $\tilde{\mathcal{H}}_k$ by Lemma \ref{defer}, since they restrict to $\tilde{\mathcal{H}}_v$ for some $v \in S \ne \emptyset$. It follows that $V_5$ is a set of representatives for $\{ \Delta \in \mathcal{H}_S\,:\, \res_v(\Delta) \in \Pic^1(C_v), \forall\,v\in S\,\}$, which is equal to $\Sel_{alg}^{(p)}(C/k)$ by Proposition \ref{Selalg}.\\

We now describe each step in more detail.
\subsubsection*{Step 1} This is straightforward.

\subsubsection*{Step 2} This is the bottleneck in the computation. Let $\mathcal{F}_S$ denote the unramified outside $S$ subgroup of $F^\times/k^\times F^{\times p}$. It fits into an exact sequence
\[ k(S,p) \to F(S,p) \to \mathcal{F}_S \to \frac{\Cl(\mathcal{O}_{k,S})}{p\Cl(\mathcal{O}_{k,S})} \to \frac{\Cl(\mathcal{O}_{F,S})}{p\Cl(\mathcal{O}_{F,S})}\,.\] For a derivation of this sequence and a description of how to compute $\mathcal{F}_S$ see \cite[12.8]{PoonenSchaefer}. 

\subsubsection*{Step 3} Since we have already computed $k(S,p)$, this can be accomplished very quickly using linear algebra over $\F_p$.

\subsubsection*{Step 4} Using Hensel's Lemma it is relatively easy to show that $\mathcal{H}_v$ is finite and that the maps $\Phi_v:\Pic^1_v(C) \to \mathcal{H}_v$ are locally constant. Moreover, one can determine the size of the image by considering the factorization of the $p$-division polynomial of the Jacobian over $k_v$ (see for example \cite[Propsition 3.8]{Schaefer} or \cite[Proposition 2.4]{SchJAC}).

To compute the local image it thus suffices to find the images of sufficiently many independent points.  It is actually easier to determine independence by considering the images in $\mathcal{H}_v$. This is valid since the descent map is injective. Moreover, since the descent map is affine it suffices to find a set of images in $\mathcal{H}_v$ which span a space of the appropriate dimension. In practice we simply compute the images of random points until their images generate a large enough space.

\subsubsection*{Step 5} Having computed $\mathcal{F}_S$ and $F_v^\times/k_v^\times F_v^{\times p}$ this can be accomplished using linear algebra over $\F_p$.

\subsubsection*{Step 6} Extracting the $p$-th roots is straightforward (if a bit costly - it is here that the degree of $H_2$ becomes a problem). By `modulo the equivalence...' we mean that we keep one $(\delta,\eps)$ in each equivalence class. To determine equivalence, one needs to determine $(\mu_p(\bF)/\mu_p)^{G_k}$ and its image under $\partial_2$.

\subsubsection*{Step 7} This is accomplished as in step (5).

\section{Examples}
\label{Examples}

\subsection{An example with $\Sha[3^\infty] = \Z/9\Z\times \Z/9\Z$.}

As a first example, let us consider the smallest elliptic curve over $\Q$ (ordered by conductor) with analytic order of $\Sha$ divisible by $81$. This is the curve \[ E : y^2 + y = x^3 - x^2 - 14556197783x - 675953651051907\] of conductor $5075$ (labelled {\tt 5075d3} in Cremona's database \cite{CremonaDB}). One can show that $E(\Q) = 0$, using either a $2$-descent or analytic means to determine the rank.

A $3$-descent on $E$ produces four plane cubic curves, each of which represents an inverse pair of nontrivial elements in $\Sel^{(3)}(E/\Q)$, so \[\Sel^{(3)}(E/\Q) \simeq \Sha(E/\Q)[3] \simeq \Z/3\Z\times\Z/3\Z\,,\] and each cubic is a counter-example to the Hasse principle. Since the analytic order of $\Sha(E/\Q)$ is $81$, we expect that $\#\Sel^{(3)}(C/\Q)=9$ for each cubic $C$. This is confirmed by performing a second $3$-descent. Note that since the order of $3\Sha(E/\Q)[9]/\Sha(E/\Q)[3]$ is a square it suffices to do the computation for a single cubic. Each of the $9$ elements computed in $\Sel_{alg}^{(3)}(C/\Q)$ correspond to a pair of inverse elements of order $9$ in $\Sel^{(9)}(E/\Q) \simeq \Sha(E/\Q)[9] \simeq \Z/9\Z \times \Z/9\Z$.

Alternatively, we may appeal to the isogeny invariance of the Birch and Swinnerton-Dyer conjecture (see \cite{CasselsVIII}). Namely the conjecture either correctly predicts $\ord_3(\Sha(E/\Q))$ for every curve in the isogeny class, or for none of them. In this example the other two curves in the isogeny class, {\tt 5075d1} and {\tt 5075d2}, are predicted to have $\Sha$ of order $1$ and $9$, respectively. This may be verified by first and second $3$-descents, and shows in addition that there are no elements of order $27$ in $\Sha(E/\Q)$.\\

\begin{Remark}
There is a degree $9$ isogeny between $E$ and {\tt 5075d1}. As the anonymous referee astutely noted, this gives a more direct proof of the fact that there are no higher order elements in $\Sha$. Namely multiplication by $9$ factors through the Shafarevich-Tate group of {\tt 5075d1} which is trivial.
\end{Remark}

\subsection{An example with irreducible mod $3$ representation}
The first elliptic curve over $\Q$ (ordered by conductor) with irreducible mod $3$ representation and analytic order of $\Sha$ divisible by $81$ has Cremona reference {\tt 15675f1} and minimal Weierstrass model \[ E : y^2 + y = x^3 - x^2 - 9002708x - 10393995307\,. \] A second $3$-descent shows that the $\Sha[9] \simeq \Z/9\Z\times\Z/9\Z$. Using the method described in Section \ref{InverseOfTheDescentMap} we can produce models for the elements of order $9$ as genus one normal curves of degree $9$ in $\PP^8$. In the appendix we give $27$ symmetric matrices which corresponding to $27$ quadratic forms which give such a model. The curve defined by the vanishing of the quadratic forms is everywhere locally solvable, yet has no rational points over any number field of degree indivisible by $9$. To our knowledge this is the largest prime power to date for which such an example has been produced\footnote{One can produce examples of order $12$ by combining examples of orders $3$ and $4$ using the method of \cite{Fisher6and12}.}.

\subsection{An example with $\Sha = \Z/12\Z\times \Z/12\Z$}
As a final example we offer the following theorem which, in addition to some very deep results in the direction of the Birch and Swinnerton-Dyer conjecture, brings to bear many of the currently available algorithms for explicit descents on curves of genus one.

\begin{Theorem}
The full Birch and Swinnerton-Dyer conjecture holds for the elliptic curves
\begin{align*}
E : y^2 &= x^3 + 7^3 \cdot 61^3\cdot 97^4\,,\text{ and}\\
E' : y^2 &=x^3 - 3^3\cdot 7^3 \cdot 61^3\cdot 97^4
\end{align*} defined over $\Q$. In particular, $\Sha(E/\Q)\,\simeq\Sha(E'/\Q) \simeq \Z/12\Z\times \Z/12\Z$.
\end{Theorem}

The hard part of the proof is taken care of by the existing partial results in the direction of BSD. The role of descent is to compute the $p$-primary parts of the Shafarevich-Tate groups at the primes $2$ and $3$. Note that these curves are related by the $3$-isogeny
\[ h : E \ni (a,b) \mapsto \left(\frac{a^3+2^27^3  61^3 97^4}{a^2},\frac{a^3b - 2^37^3  61^3 97^4b}{a^3}\right) \in E'\,.\] So the validity of BSD for either curve implies its validity for the other.

One can check that the values of the $L$-series of $E$ and $E'$ at $s=1$ are (equal and) approximately $5.5542$. Results of Coates and Wiles then imply that the Mordell-Weil groups are finite \cite{CoatesWiles}. One easily checks that there is no nontrivial torsion on either, so the Mordell-Weil groups are trivial. The predicted orders of $\Sha(E/\Q)$ and $\Sha(E'/\Q)$ are the numbers 
\begin{align*}
\Sha_{an}(E) &= \frac{L_E(1)}{\Omega(E) \cdot \prod_{p | \Delta(E)} C_p(E)}\,\text{  and}\\
\Sha_{an}(E') &= \frac{L_{E'}(1)}{\Omega(E') \cdot \prod_{p | \Delta(E')} C_p(E')}\,,
\end{align*}
where $L_\mathcal{E}(s)$ is the Hasse-Weil $L$-function associated to $\mathcal{E}$, $\Omega(\mathcal{E})$ is the real period, $C_p(\mathcal{E})$ denotes the Tamagawa number of $\mathcal{E}$ at $p$ and $\Delta(\mathcal{E})$ is the discriminant (note that the regulators and torsion subgroups are trivial for both curves). The real period of $E$ is $\Omega(E)\approx 0.0096427$ and the only Tamagawa number not equal to $1$ is $C_7(E)=4$. This gives 
\[ \Sha_{an}(E) \approx \frac{5.5542}{(0.0096427)\cdot4} \approx 144 \]
The real period of $E'$ satisfies $\Omega(E) = 3\cdot\Omega(E')$ and the nontrivial Tamagawa numbers are $C_3(E')=3$ and $C_7(E')=4$. Thus $\Sha_{an}(E')\approx 144$ as well.

It is known that $\Sha_{an}(E)$ and $\Sha_{an}(E')$ are rational numbers of (explicitly) bounded denominator, so taking the computations to sufficiently high precision we conclude that they are in fact equal to $144$. Rubin's result \cite{Rubin91} then implies that $\Sha(E/\Q)[p] = 0$ for all primes $p$ not dividing $\Sha_{an}(E)$ or the order of the group of units in the ring of integers of the field of complex multiplication. The same holds for $\Sha(E'/\Q)$. These curves have CM by $\sqrt{-3}$, so we conclude that both Shafarevich-Tate groups are annihilated by some power of $6$.

After applying these deep results, we are left only with the task of computing the $2$- and $3$-primary parts of $\Sha(E/\Q)$ and $\Sha(E'/\Q)$. Since the validity of BSD for a given elliptic curve is actually a property of its isogeny class, it will suffice to perform the computations for either curve. We describe the computations for $E$. The computations for $E'$ are similar (and equally feasible).

\subsubsection*{The $2$-primary part} One needs explicit first and second $2$-descents to produce models for the elements of order dividing $4$ and then a third $2$-descent to show that there are no elements of order $8$. The $2$-descent on $E$ yields models for the $3$ nontrivial elements of $\Sha(E/\Q)[2]$ as double covers of $\PP^1$:
\begin{align*}
C_1: u_3^2 &= 130174u_1^4 - 71004u_1^3 - 426024u_1^2 + 2011780u_1 - 390522\,,\\
C_2: u_3^2 &= 11834u_1^4 + 260348u_1^3 - 710040u_1^2 + 1372744u_1 + 3999892\,,\\
C_3: u_3^2 &= 5917u_1^4 + 29585u_1^3 - 177510u_1^2 + 804712u_1 + 562115\,.
\end{align*} For each $C_i$ a second $2$-descent will produce a pair of quadric intersections, each representing a pair of inverse elements of order $4$ in $\Sha(E/\Q)$. For example, $\Sel^{(2)}(C_3/\Q)$ is of order $4$ and represented by the two curves (for each their are two choices for the covering map) {\small
\begin{align*}
D_1 &= \left\{
    \begin{array}{ccc}
    2z_1^2 + 14z_1z_2 - 3z_2^2 + 4z_1z_3 - 2z_2z_3 + 5z_3^2 + 8z_1z_4 +
        2z_2z_4 - 8z_3z_4 - 15z_4^2 = 0\\
    24z_1^2 + 8z_1z_2 - 22z_2^2 + 36z_1z_3 + 18z_2z_3 + 63z_3^2 - 54z_1z_4 -
        24z_2z_4 + 42z_3z_4 + 14z_4^2 = 0
\end{array}
  \right\} \subset \PP^3\,,\\
D_2 &= \left\{
    \begin{array}{ccc}
    3z_1^2 + 2z_1z_2 + 3z_2^2 + 6z_1z_3 - 2z_2z_3 - 8z_3^2 + 6z_1z_4 +
        24z_2z_4 - 13z_4^2 =0\\
    6z_1^2 + 86z_1z_2 - 20z_2^2 - 18z_1z_3 + 2z_2z_3 + 13z_3^2 - 18z_1z_4 -
        22z_2z_4 - 6z_3z_4 - 42z_4^2 = 0
\end{array}
  \right\} \subset \PP^3\,.
\end{align*}} 
One then uses Stamminger's method for third $2$-descent which shows that none of the elements of order $4$ lift to elements of order $8$. It follows that the $2$-primary part is $\Sha(E/\Q)[2^\infty]\simeq\Z/4\Z\times\Z/4\Z$.

\subsubsection*{The $3$-primary part}
For this we can make use of the $3$-isogeny. A $3$-isogeny descent (as described in \cite{SchaeferStoll}) computes that $\Sel^{(h)}(E/\Q) \simeq \Sel^{(h')}(E'/\Q) \simeq \Z/3\Z$. Since the Mordell-Weil groups are trivial these Selmer groups are isomorphic to the corresponding torsion subgroups of the Shafarevich-Tate groups. The exact sequence 
\[ \text{\small $0 \to \frac{E'(\Q)[h']}{h(E(\Q)[3])} \to \Sel^{(h)}(E/\Q) \to \Sel^{(3)}(E/\Q) \stackrel{h}{\to} \Sel^{(h')}(E'/\Q) \to \frac{\Sha(E'/\Q)[h']}{h\left(\Sha(E/\Q)[3]\right)} \to 0$ } \] reduces to
\begin{align}
\label{SelIsogSeq}
 0 \to \Sel^{(h)}(E/\Q) \to \Sel^{(3)}(E/\Q) \to \Sel^{(h')}(E'/\Q) \to 0\,,
\end{align} which splits since $\Sel^{(3)}(E/\Q)$ is $3$-torsion. We conclude that $\Sha(E/\Q)[3] \simeq \Z/3\Z\times\Z/3\Z$.

The $3$-isogeny descent (implemented in \magma) is also explicit in that it produces the projective plane cubic
\[ C : u_1^3 + 4u_2^3 + 4017643u_3^3 = 4u_1^2u_2 + 3u_1u_2^2 \] representing the pair of nontrivial elements in $\Sel^{(h)}(E/\Q)$. This also represents an inverse pair of nontrivial elements in $\Sel^{(3)}(E/\Q)$. In order to show that $\Sha(E/\Q)[3^\infty] \simeq \Z/3\Z\times\Z/3\Z$ it will be enough to show that $\Sel^{(3)}(C/\Q) = \emptyset$ (since $\Sha(E/\Q)[3^\infty]$ is finite, its order is a square).

For this we do a second $3$-descent. The reducibility of $E[3]$ translates into a splitting of the flex algebra. We find that $F$ is isomorphic to the product of the cubic and sextic number fields with defining polynomials
\begin{align*}
f(t) &= t^3 - 4t^2 - 3t + 4\text{ and}\\
f(t) &= t^6 + 2408704t^3 + 5533080062500
\end{align*}
and the set of bad brimes is $S = \{ 2,3,5,7,61,97\}$. Despite the splitting, computation of $F(S,3)$ takes a couple hours of processor time. Having accomplished that however the remaining computations are very fast. In fact, the computation can be completed without ever using $H_2$. The image of the $3$-Selmer set under $\pr_1$ is contained in the set of all $\delta \in F^\times/\Q^\times F^{\times 3}$ such that
\begin{enumerate}
\item $\delta$ is unramified outside $S$,
\item $N_{F/\Q}(\delta) \equiv 7^261^297 \Mod \Q^{\times 3}$ and  
\item $\forall\,v\in S, \res_v(\delta) \in \pr_1\Phi(C(\Q_v))$\,.
\end{enumerate}  which turns out to be empty. It follows that the $3$-Selmer set of $C$ is empty and thus that the $3$-primary part of $\Sha(E/\Q)$ is isomorphic to a product of two cyclic groups of order $3$.\\

\vfill 
\pagebreak
\section{An element of order $9$ in $\Sha$ for the curve \text{\tt 15675f1}}
{\tiny
\begin{align*}
&\begin{bmatrix}
0& 0& 0&-1& 1& 1& 0&-1& 1\\
& 0&-1&-1& 0& 0& 0& 1& 1\\
&& 0& 0& 1& 0& 0&-2& 1\\
&&& 2&-1&-1&-1& 3& 0\\
&&&& 2& 2&-1& 2& 1\\
&&&&& 0&-1&-1& 2\\
&&&&&&-2&-2&-1\\
&&&&&&& 4& 0\\
&&&&&&&& 2\\
\end{bmatrix}
&&\begin{bmatrix}
 0& 1& 0& 0& 0& 0& 1&-3& 1\\
& 0&-1& 1& 1& 1& 1&-1& 0\\
&&-2& 2&-1& 1& 1& 0& 0\\
&&& 2& 0& 0& 1&-3&-1\\
&&&& 2& 1& 1&-1& 1\\
&&&&& 2& 0& 0& 2\\
&&&&&& 2& 2& 2\\
&&&&&&&-2&-1\\
&&&&&&&& 2\\
\end{bmatrix}\\
&\begin{bmatrix}
 0& 0& 1& 1& 0&-1& 1& 3&-2\\
& 2& 1& 1& 1& 0& 0& 0&-1\\
&& 0& 0& 0&-2& 1&-1&-1\\
&&& 0& 0& 0& 0& 1& 1\\
&&&& 0&-2& 0& 2&-1\\
&&&&&-4&-1& 0& 2\\
&&&&&&-2& 0& 0\\
&&&&&&& 2& 1\\
&&&&&&&& 4\\
\end{bmatrix}
&&\begin{bmatrix}
 0& 0& 1& 0& 1& 1& 0& 1& 0\\
& 0&-1& 2& 0& 1& 1& 0& 0\\
&& 0& 2&-2&-1&-1& 2& 2\\
&&& 2& 1& 0& 1& 2&-2\\
&&&& 2& 1& 0& 0& 1\\
&&&&& 0&-1& 1& 2\\
&&&&&& 2& 0& 1\\
&&&&&&& 2& 3\\
&&&&&&&& 2\\
\end{bmatrix}\\
&\begin{bmatrix}
 0& 0& 0& 0& 0& 0& 2&-3& 0\\
 & 0&-1& 0& 0& 0& 0&-2&-1\\
 && 0& 1& 0& 0& 2&-1& 1\\
 &&& 0& 1& 1& 1& 0& 0\\
 &&&& 0&-1& 1&-1& 1\\
 &&&&&-2& 1& 4& 2\\
 &&&&&& 2&-3&-1\\
 &&&&&&& 0&-1\\
 &&&&&&&& 0\\
\end{bmatrix}
&&\begin{bmatrix}
 0& 1& 0& 0& 1& 0& 0& 1&-1\\
 & 0& 1& 1& 1& 2& 1& 0&-1\\
 &&-2&-1& 1& 0& 0&-3&-1\\
 &&& 0&-1& 0&-2& 2& 2\\
 &&&& 4& 3& 1& 0&-1\\
 &&&&& 0& 0& 0& 0\\
 &&&&&& 0&-3&-1\\
 &&&&&&& 0& 1\\
 &&&&&&&& 0\\
\end{bmatrix}\\
&\begin{bmatrix}
 0& 1& 1& 0& 1& 0& 0& 0& 0\\
 & 0& 0& 2& 1& 2& 1& 0& 0\\
 &&-2& 1& 0& 1& 0& 3& 1\\
 &&& 2&-1&-1&-1&-2& 0\\
 &&&& 4& 2& 1& 1& 1\\
 &&&&& 0&-1& 2& 1\\
 &&&&&& 4&-2& 0\\
 &&&&&&& 0&-1\\
 &&&&&&&& 0\\
\end{bmatrix}
&&\begin{bmatrix}
 0& 0& 0& 0& 1& 1& 1&-1&-1\\
 & 0& 0& 1& 0& 1& 2&-2&-1\\
 && 0& 1&-2&-2&-1& 1& 0\\
 &&&-2& 2& 2& 0& 1& 0\\
 &&&& 2& 1& 0&-2&-1\\
 &&&&& 0&-1& 1& 1\\
 &&&&&& 0& 0& 4\\
 &&&&&&& 0&-2\\
 &&&&&&&& 2\\
\end{bmatrix}\\
&\begin{bmatrix}
 0& 0& 1& 0& 1& 0& 0& 0&-1\\
 & 0& 1& 0& 0& 0& 0& 1& 1\\
 &&-2& 0& 1&-1& 1&-1& 0\\
 &&& 2& 0& 1&-1&-1& 1\\
 &&&& 2& 2&-1& 0&-1\\
 &&&&& 0&-1&-3&-3\\
 &&&&&&-2& 0& 0\\
 &&&&&&& 6&  3\\
 &&&&&&&&-2\\
\end{bmatrix}
&&\begin{bmatrix}
 0& 0& 0& 1& 0& 1& 0& 0&-1\\
 & 0& 1&-1& 1& 0& 0& 3& 1\\
 && 0& 0& 0& 0& 1& 1& 1\\
 &&& 2&-1&-1& 0& 1& 0\\
 &&&& 2& 2& 1& 1& 1\\
 &&&&& 4& 0&-2& 2\\
 &&&&&& 2& 1&-1\\
 &&&&&&&-2&-2\\
 &&&&&&&&-6\\
\end{bmatrix}\\
&\begin{bmatrix}
 0& 1& 1&-1&-1&-1& 0& 1& 1\\
 & 0& 0& 3& 0& 1& 0&-1&-1\\
 && 2& 0&-1& 1& 0&-1& 1\\
 &&&-2& 0&-1& 1&-3& 2\\
 &&&&-2&-2& 0& 1& 1\\
 &&&&&-2& 1& 1& 2\\
 &&&&&&-2& 2& 3\\
 &&&&&&&-2& 3\\
 &&&&&&&& 0\\
    \end{bmatrix}
&&\begin{bmatrix}
 2& 0& 2& 0& 0& 0& 1& 0& 2\\
 & 0& 1& 0& 0& 0& 0& 1& 1\\
 && 6 & 1& 3& 2& 1& 3& 1\\
 &&& 0&-1&-2&-1& 1& 1\\
 &&&& 0& 0& 0& 2& 1\\
 &&&&&-2& 1& 1& 0\\
 &&&&&& 0& 0& 0\\
 &&&&&&& 0&-1\\
 &&&&&&&& 2\\
    \end{bmatrix}\\
&\begin{bmatrix}
 0& 0& 0& 0& 1& 0& 1& 0&-2\\
 & 0& 2& 0& 1& 2& 0& 0&-2\\
 && 2&-1&-2&-2& 1& 1& 0\\
 &&&-2& 1& 1& 1&-3& 2\\
 &&&& 2&-1& 0& 2& 0\\
 &&&&&-4& 0& 1& 0\\
 &&&&&& 0&-1& 3\\
 &&&&&&& 0& 0\\
 &&&&&&&& 2\\
    \end{bmatrix}
&&\begin{bmatrix}
 0& 0& 2& 0& 0& 1&-1& 0&-1\\
 & 0& 1& 1&-1& 0& 0&-1& 1\\
 && 0& 0&-1& 2& 2&-2& 1\\
 &&& 4& 1& 2& 0&-3& 1\\
 &&&&-4&-2&-1& 0& 2\\
 &&&&& 0& 1&-1&-2\\
 &&&&&& 0& 1& 1\\
 &&&&&&& 0& 2\\
 &&&&&&&& 2\\
    \end{bmatrix}\\
 \end{align*} 
 \vfill
 \pagebreak
 \begin{align*}
&\begin{bmatrix}
 2& 0&-1& 2& 0& 1& 1& 1& 0\\
 & 0& 0& 0& 0& 0& 0& 1& 0\\
 && 2& 0& 0& 0&-1&-3& 2\\
 &&&-2& 0&-3& 2& 5&1\\
 &&&& 0&-1&-1& 2&-1\\
 &&&&&-2&-1& 0& 2\\
 &&&&&& 0&-1& 0\\
 &&&&&&&-2& 0\\
 &&&&&&&& 0\\
    \end{bmatrix}
&&\begin{bmatrix}
 0& 1& 0&-1&-1&-2& 1&-1& 3\\
 & 0& 0& 1& 1& 0& 0& 0& 0\\
 && 0& 0& 1&-1&-1& 2& 1\\
 &&&-4& 1&-1& 0&-4&-2\\
 &&&& 0& 1& 1&-1& 1\\
 &&&&& 4& 0& 0& 3\\
 &&&&&& 2& 1& 1\\
 &&&&&&&-2& 0\\
 &&&&&&&&-2\\
    \end{bmatrix}\\
&\begin{bmatrix}
 0& 0& 1&-1& 0& 1&-1& 0& 0\\
 & 0& 0& 1&-1& 1& 2&-1& 1\\
 && 2& 0& 0& 1&-1& 3& 3\\
 &&& 0&-2&-5&0& 2& 0\\
 &&&& 0& 2&-2& 1& 1\\
 &&&&& 0& 0&-1& 0\\
 &&&&&&-2& 0& 1\\
 &&&&&&& 2&-1\\
 &&&&&&&& 2\\
\end{bmatrix}
&&\begin{bmatrix}
-2& 1& 0&-2& 1& 0& 1&-1&-1\\
 & 0&-2& 2& 1& 1& 1& 1& 1\\
 &&-2& 1&-1&-2& 0&-1&-2\\
 &&&-2& 2& 3& 0&-1&-2\\
 &&&& 4& 3& 2&-1&-1\\
 &&&&& 2& 0&-1& 1\\
 &&&&&& 0& 1& 1\\
 &&&&&&&-2&-1\\
 &&&&&&&&-2\\
    \end{bmatrix}\\
&\begin{bmatrix}
 0& 0& 0& 1& 0&-1& 1&-1& 1\\
 & 0&-1& 1& 0& 0& 1& 0& 1\\
 &&-2& 1& 1&-1& 1& 3& 0\\
 &&& 2&-1&-3& 0& 2& 1\\
 &&&& 4& 3&-1&-1&-2\\
 &&&&& 2&-2& 1&-3\\
 &&&&&& 2& 1& 2\\
 &&&&&&& 0&-2\\
 &&&&&&&& 0\\
    \end{bmatrix}
&&\begin{bmatrix}
 0& 0& 1& 0&-1&-1& 0& 1& 2\\
 & 0&-1& 0& 0&-2& 0& 0& 1\\
 &&-2& 2& 1& 1& 1& 3& 0\\
 &&& 2&-2& 2&-2&-3& 0\\
 &&&& 0&-2& 2&-1&-1\\
 &&&&&-4& 0& 2&-1\\
 &&&&&& 0& 0&-3\\
 &&&&&&& 2& 1\\
 &&&&&&&& 2\\
    \end{bmatrix}\\
&\begin{bmatrix}
 2& 1& 0& 1& 0& 0& 0&-1& 3\\
 & 0& 0& 1& 0& 0& 1& 0& 2\\
 && 0& 1& 1& 2&-1& 1& 2\\
 &&& 0&-1&-4& 0& 0& 0\\
 &&&& 0& 1&-1& 1& 4\\
 &&&&& 2&-1&-2& 2\\
 &&&&&& 4& 0& 1\\
 &&&&&&& 4& 1\\
 &&&&&&&&-2\\
    \end{bmatrix}
&&\begin{bmatrix}
 2& 1& 1& 0& 2& 2&-1&-1& 0\\
 & 0& 2& 0& 0& 1& 0& 1& 3\\
 && 0& 0& 1& 2&-1& 0& 2\\
 &&& 0&-1& 0&-1&-2& 2\\
 &&&& 0& 1&-2& 2& 4\\
 &&&&& 0&-2&-1& 1\\
 &&&&&& 0& 0&-1\\
 &&&&&&& 2& 0\\
 &&&&&&&&-6\\
    \end{bmatrix}\\
&\begin{bmatrix}
 2& 1&-1& 2& 1& 0& 1&-2& 2\\
 & 0& 1&-1& 1& 0& 0& 1& 1\\
 &&-6&  2& 2& 0&-2&-1&-1\\
 &&& 0& 1& 3&-2&-1&-2\\
 &&&& 2& 3& 1&-2& 0\\
 &&&&& 4& 0&-2& 0\\
 &&&&&& 0& 1& 0\\
 &&&&&&& 0& 0\\
 &&&&&&&& 0\\
    \end{bmatrix}
&& \begin{bmatrix}
 2&-1& 0& 2& 1& 1& 1& 0& 2\\
 & 0& 0&-1& 0&-1&-1& 1& 1\\
 && 0& 2&-1&-5&0& 1& 1\\
 &&& 0& 4& 2& 2& 2&-1\\
 &&&& 0& 0&-2&-1&-1\\
 &&&&& 2&-3&-1& 0\\
 &&&&&& 2& 1& 0\\
 &&&&&&& 0&-1\\
 &&&&&&&&-2\\
    \end{bmatrix}\\
&\begin{bmatrix}
 0& 1& 1& 1& 1&-2& 2& 1& 0\\
 & 2&-1& 3& 2& 1&-3& 2& 0\\
 &&-2& 1& 1&-2& 1& 1&-1\\
 &&& 0& 2&-1& 2&-1& 3\\
 &&&& 2& 0& 0& 2&-1\\
 &&&&& 2&-2&-1& 3\\
 &&&&&& 2& 2& 0\\
 &&&&&&& 4& 4\\
 &&&&&&&&-2\\
    \end{bmatrix}
&&\begin{bmatrix}
 0& 1& 1&-2& 1& 1& 1&-1& 1\\
 & 0& 0& 0& 2& 0& 2& 1& 0\\
 && 2& 1&-3&-2& 1&-1& 4\\
 &&& 2&-1& 1& 1&-1& 0\\
 &&&& 6&  4&-1& 0& 1\\
 &&&&& 4&-3&-2& 0\\
 &&&&&&-2& 0& 4\\
 &&&&&&& 0& 0\\
 &&&&&&&&-2\\
     \end{bmatrix}\\
     &\begin{bmatrix}
 0& 0& 1& 0& 1& 1& 1&-3& 0\\
 & 0&-1& 0& 0& 1& 0&-2& 0\\
 &&-2& 2& 0& 0& 1& 3& 1\\
 &&& 2& 2& 2& 0&-1& 0\\
 &&&& 0& 0& 1&-2& 0\\
 &&&&& 0& 1& 0& 1\\
 &&&&&& 2& 1& 0\\
 &&&&&&& 0&-5\\
 &&&&&&&&-2\\
    \end{bmatrix}
\end{align*}
}

\vfill
\pagebreak

\end{document}